\documentclass[reqno, 11pt, a4paper]{amsart}
\usepackage[utf8]{inputenc}
\usepackage{amsfonts}
\usepackage{amsmath}
\usepackage{amssymb}
\usepackage{amsthm}
\usepackage{mathrsfs} 
\usepackage[dvipsnames]{xcolor}

\usepackage[colorlinks=true, linkcolor=blue, citecolor=blue]{hyperref}

\usepackage{appendix}
\usepackage{enumerate}
\usepackage[text={33pc,605pt},centering]{geometry}    

\usepackage[
    backend=biber,
    labeldate=year,
    date=year,
    natbib=true,
    url=false, 
    doi=false,
    eprint=false,
    style=numeric-comp,
    sorting=nyt,
    abbreviate=false,
    isbn=false,
    maxcitenames=9,
    maxbibnames=9,
    doi=true,
    giveninits = true]{biblatex}
\AtEveryBibitem{
  \clearlist{language}
  \clearfield{labelmonth}
}
\newtheorem{theorem}{Theorem}[section]
\newtheorem{lemma}[theorem]{Lemma}
\newtheorem{prop}[theorem]{Proposition}
\newtheorem{assumption}[theorem]{Assumption}

\theoremstyle{definition}
\newtheorem{definition}[theorem]{Definition}
\newtheorem{example}[theorem]{Example}

\theoremstyle{remark}
\newtheorem{remark}[theorem]{Remark}
\newtheorem{question}[theorem]{Question}

\numberwithin{equation}{section} 

\newcommand{\norm}[1]{\left \lVert  #1 \right \rVert}
\newcommand{\abs}[1]{\left\lvert #1 \right\rvert}

\newcommand{\vertiii}[1]{{\left\vert\kern-0.25ex\left\vert\kern-0.25ex\left\vert #1 
    \right\vert\kern-0.25ex\right\vert\kern-0.25ex\right\vert}}

\newcommand{\Z}{\mathbb{Z}}
\newcommand{\C}{\mathbb{C}}
\newcommand{\R}{\mathbb{R}}
\newcommand{\N}{\mathbb{N}}

\newcommand{\calF}{\mathcal{F}}

\newcommand{\dd}{\mathop{}\!\mathrm{d}}

\newcommand{\e}{\mathrm{e}}
\newcommand{\jap}[1]{\langle #1 \rangle}

\newcommand{\Cluster}{\mathcal{C}}

\title[Stationary measures of the critical Ornstein--Uhlenbeck process]{On the stationary measures of the critical Ornstein--Uhlenbeck process}
\author[J.-D. Deuschel]{Jean-Dominique Deuschel}
\address{Technische Universität Berlin}
\email{deuschel@math.tu-berlin.de}
\author[J. K\"oppl]{Jonas K\"oppl}
\address{Technische Universität Braunschweig}
\email{jonas.koeppl@tu-braunschweig.de}
\author[Y. Steenbeck]{Yannic Steenbeck}
\address{Technische Universität Braunschweig}
\email{yannic.steenbeck@tu-braunschweig.de}
\author[T. Worschech]{Thomas Worschech}
\address{Technische Universität Berlin}
\email{taworschech@gmail.com}

\date{\today}
\begin{document}
\begin{abstract} 
 We study linear and symmetric diffusion processes on $\R^{\Z^d}$ that can be seen as the Langevin dynamics of the (inhomogeneous) harmonic crystal for given conductances, with a particular focus on their stationary distributions. Despite the linearity of the interaction, we exhibit a rich variety of behaviours, depending on the disorder encoded by the conductances. Our main results provide essentially sharp criteria on the conductances ensuring that every stationary measure is reversible. We also provide examples for which these criteria fail and where either no stationary measures exist or reversible and non-reversible stationary measures coexist.  
 Additionally, we show that the linear system can even exhibit  non-trivial time-periodic behaviour and  provide a spectral characterisation of the occurrence of such oscillations. 
 The case of deterministic conductances is complemented by a study of the case of random conductances under quite general moment assumptions. 
\end{abstract}
\maketitle
\setcounter{tocdepth}{1}
\tableofcontents

\section{Introduction}
We study linear and symmetric diffusion processes on $\R^{\Z^d}$ given in terms of a coupled system of real-valued SDEs of the form 
\begin{equation}\label{sde-intro}
    \dd\phi_t(x) = \sum_{y\sim x}\omega(x,y)(\phi_t(y)-\phi_t(x))\dd t + \sqrt{2}\dd B_t(x), \quad x \in \Z^d,  
\end{equation}
where $\omega = (\omega(x,y))_{x,y\in\Z^d}$ is a set of symmetric nearest-neighbour conductances and $(B_t(x))_{x\in \Z^d, t\geq0}$ are independent standard Brownian motions. Our assumptions on $\omega$ will ensure that the associated variable speed random walk does not explode in finite time.
These processes $\{\phi_t(x) \colon x\in \Z^d, t\geq 0\}$ are also known as \textit{critical Ornstein--Uhlenbeck processes} and have for example been studied in the context of aging phenomena, both in the physics \cite{parisi} and mathematics literature \cite{dembo_deuschel}. 

Of particular concern to us will be the analysis of the stationary distributions for the SDE \eqref{sde-intro}. Aside from the connection which these diffusions have with the (inhomogeneous) discrete Gaussian Free Field in statistical mechanics, we believe that the study of such diffusions and their possible long-time behaviour is interesting on purely mathematical grounds. Indeed, as far as we know, the ergodic theory of infinite-dimensional diffusions is still poorly understood, especially in the regime of phase coexistence, i.e.\, when there may be more than one stationary distribution. 
Despite the linearity of \eqref{sde-intro}, the stationary measures can exhibit a surprisingly rich variety of different behaviours, depending on the growth of the conductances, or rather their local sum $\pi(x) := \sum_{y\sim x}\omega(x,y)$, as a function of $\norm{x}$.

Before we become a bit more rigorous and set up the precise mathematical terrain on which we will proceed, let us briefly highlight some of the implications of our main results. In particular, our results show an essentially sharp transition at quadratic growth of the conductances. 
\begin{enumerate}[i.]
    \item If the conductances satisfy a subquadratic growth bound and the random walk is transient, then \eqref{sde-intro} admits stationary measures and every stationary measure  is also a reversible measure and in particular a Gibbs measure. 
    \item There are non-degenerate conductances $\omega$ that grow quadratically such that \eqref{sde-intro} admits both reversible and non-reversible stationary measures. 
    \item There are non-degenerate conductances $\omega$ that grow quadratically such that \eqref{sde-intro} exhibits non-trivial time-periodic behaviour, i.e.\, there is an initial distribution $\nu_0$ and a time-period $\tau>0$ such that if we initialise \eqref{sde-intro} with $\phi_0 \sim \nu_0$, then $\phi_t \not\sim \nu_0$ for all $t\in(0,\tau)$ and $\phi_\tau \sim \nu_0$. 
    \item For ergodic translation-invariant random conductances we establish a moment condition that ensures that almost surely every stationary measure is also a reversible measure and in particular a Gibbs measure. 
\end{enumerate}
Item $i.$ is analogous to previous results for interacting particle systems with compact local spin spaces in $d=1,2$ proved in \cite{holley_one_1977,holley_diffusions_1981} but our result also works in any dimension $d\geq 3$. This Gibbsian characterisation of the stationary measures is what one would typically expect, even though it is highly non-trivial to prove without additionally assuming shift-invariance. 

By contrast, Items $ii.$ and $iii.$ demonstrate that this picture breaks down at quadratic growth. To our knowledge, Item $iii.$ provides the first example of a reversible interacting particle system exhibiting time-periodic behaviour. Moreover, we give a spectral characterisation of the occurrence of such periodic behaviour.

Taken together, these results identify an essentially sharp transition in the long-time behaviour of the dynamics. Below quadratic growth all stationary measures are reversible, whereas at quadratic growth one may encounter non-reversible stationary states and non-trivial time-periodic behaviour. 

\subsection*{Organisation of the manuscript}
In Section \ref{section:setting-results} we make the mathematical setup precise and state our main results before discussing their relation to the existing literature in Section \ref{section:literature}. We then proceed to explain the main strategy of the proofs in Section \ref{section:proof-strategy} before providing all technicalities in the remaining sections. 

\section{Setting and main results}\label{section:setting-results}
Consider a set of symmetric nearest-neighbour conductances $(\omega(x,y))_{x,y\in \Z^d}$ on $\Z^d$ and the corresponding variable speed random walk (VSRW) $(X_t)_{t\geq 0}$ with generator $-\Delta_\omega$ defined by
\begin{equation}
    (\Delta_\omega u)(x) := \sum_{y \sim x}\omega(x,y)[u(x)-u(y)]. 
\end{equation}
We will allow for vanishing conductances but always assume that the natural bond percolation model associated to the set of conductances admits a unique unbounded component and will denote this component by $\mathcal{C}(\omega)$. Of course, if the conductances are strictly positive then $\mathcal{C}(\omega) = \Z^d$. 
In general, the variable speed random walk may have a finite (random) explosion time $\zeta$, so that we set 
\begin{equation}
    p^\omega_t(x,y) := \mathbf{P}_x^\omega[X_t = y, \zeta > t], 
\end{equation}
where $\mathbf{P}_x^\omega$ denotes the law of the variable speed random walk on the conductances $\omega$ started in $x\in \mathcal{C}(\omega)$. The expectation with respect to $\mathbf{P}_x^\omega$ will be denoted by $\mathbf{E}_x^\omega$. Moreover, we will denote the Green kernel of the random walk generated by $-\Delta_\omega$ by $G^\omega(\cdot, \cdot)$, i.e.\, 
\begin{align*}
    G^\omega(x,y) = \mathbf{E}^\omega_x\left[\int_0^\infty \mathbf{1}_{\{X_s = y\}}\dd s\right], \quad x,y \in \mathcal{C}(\omega). 
\end{align*}
This is the expected amount of time that the random walk spends at site $y$ when started in $x$. 
We will be interested in the stationary measures of the infinite-dimensional system of SDEs given by 
\begin{equation}\label{sde}
    d\phi_t(x) = \sum_{y\sim x}\omega(x,y)(\phi_t(y)-\phi_t(x))\dd t + \sqrt{2}\dd B_t(x), \quad x \in \mathcal{C}(\omega),  
\end{equation}
where $(B_\cdot(x))_{x\in \mathcal{C}(\omega)}$ are independent standard Brownian motions. 
Denote by $\mathbb{F} := (\mathcal{F}_t^B)_{t \geq 0}$ the filtration generated by this collection of Brownian motions. 
For a fixed initial condition $\phi_0: \mathcal{C}(\omega) \to \R$, the variation-of-constants formula yields the formal solution 
\begin{equation}\label{sde-formal-solution}
    \phi_t(x) = (P^\omega_t \phi_0)(x) + \sqrt{2}\sum_{y \in \mathcal{C}(\omega)}\int_0^t p^\omega_{t-s}(x,y)\dd B_s(y), \quad x \in \mathcal{C}(\omega), t\geq0, 
\end{equation}
where $(P^\omega_t)_{t\geq 0}$ is ``the" semigroup of the variable speed random walk, for now with unspecified domain. To fix notation, for an initial condition $\phi$, we will denote expectation with respect to the process defined by \eqref{sde} by $\mathbb{E}_\phi$. 
To rigorously make sense of both of the terms in \eqref{sde-formal-solution}, we will work with the following function spaces on $\mathcal{C}(\omega)$. 
The \emph{Schwartz space} \(\mathcal{S}(\mathcal{C}(\omega))\) is defined by
    \begin{align*}
        \mathcal{S} := \mathcal{S}(\mathcal{C}(\omega))
        := \bigcap_{N \geq 0} H_N,
    \end{align*} 
    where we use the weighted $\ell^2$-spaces
    \begin{equation*}
        H_N 
        = H_N(\Cluster(\omega))
        := \ell^2(\mathcal{C}(\omega), \langle x \rangle^{2N}) := \{ f: \mathcal{C}(\omega) \to \R\colon \Vert f \Vert_{H_N}^2 < \infty \},
    \end{equation*}
    with norms 
    \begin{align*}
        \Vert f \Vert_{H_N}^2
        := \Vert f \Vert_{\ell^{2}(\mathcal{C}(\omega), \langle x \rangle^{2N})}^2
        = \sum_{x \in \Cluster(\omega)} \langle x \rangle^{2N} \, \vert f(x) \vert^2. 
    \end{align*}
    Here and from now on we will use the notation $\langle x \rangle := (1+\abs{x}^2)^{1/2}$. 
    The topology on \(\mathcal{S}\) is such that a sequence \((f_n)_n\) of Schwartz functions converges to a Schwartz function \(f\) in \(\mathcal{S}\), i.e.\, \(f_n \xrightarrow[n \to \infty]{\mathcal{S}} f\),  if and only if \(\Vert f - f_n \Vert_{H_{N}} \xrightarrow[n \to \infty]{} 0\) for \textit{all} \(N \geq 0\).
    We denote its dual, the space of \emph{tempered distributions}, by \(\mathcal{S}'(\mathcal{C}(\omega))\) or just $\mathcal{S}'$.
Classically, these two spaces can be characterised as follows. 
\begin{prop}[Characterisation of \(\mathcal{S}, \mathcal{S}'\)]
    We have:
    \begin{enumerate}[i.]
        \item The Schwartz space consists of functions which decay faster than any polynomial grows
        \begin{align*}
            \mathcal{S}
            = \{f \colon \mathcal{C}(\omega) \to \mathbb{R} \,\colon\, \sup_{x} \, \langle x\rangle^{2N} \vert f(x) \vert < \infty \text{ for all } N \geq 0  \}.
        \end{align*}

        \item The tempered distributions are functions which grow at most polynomially
        \begin{align*}
            \mathcal{S}'
            = \{f \colon \mathcal{C}(\omega) \to \mathbb{R} \,\colon\, \text{There exists } N \geq 0 \text{ such that } \sup_{x} \, \langle x\rangle^{-2N} \vert f(x) \vert < \infty  \}
        \end{align*} and the dual pairing \(\langle \cdot, \cdot \rangle_{\mathcal{S}', \mathcal{S}}\) is given by
        \begin{align*}
            \langle \xi, u \rangle_{\mathcal{S}', \mathcal{S}}
            = \sum_{x} \xi(x) u(x),
        \end{align*} where the sum converges absolutely.
    \end{enumerate}
\end{prop}

On $\mathcal{S}'$, we will also use the dual (semi)norms 
\begin{align*}
    \Vert \xi \Vert_{H_{-N}}^2
    := \sum_{x \in \Cluster(\omega)} \langle x \rangle^{-2N} \vert \xi(x)\vert^2.
\end{align*}
Observe that for every \(\xi \in \mathcal{S}'\) there is some \(N \in \mathbb{N}\) such that \(\Vert \xi \Vert_{H_{-N}} < \infty\). It is also not hard to see that for \(\xi \in \mathcal{S}'\) and all \(N \in \mathbb{N}_0\) it holds that
\begin{align}\label{equation:H_minus_N_norm_by_duality}
    \Vert \xi \Vert_{H_{-N}}
    = \sup_{0 \neq u \in \mathcal{S}} \frac{ \langle \xi, u \rangle_{\mathcal{S}', \mathcal{S}} }{\Vert u \Vert_{H_N}}.
\end{align}

\subsection{Well-posedness}
In order to make the formal solution \eqref{sde-formal-solution} well-defined for  initial conditions $\phi_0 \in \mathcal{S}'$, we will make the following assumptions on the conductances. 
\begin{assumption}[Quadratic growth]\label{assumption:at_most_quadratically_growing}
    The conductances \((\omega(x, y))_{x, y \in \Z^d}\) have at most quadratic growth, i.e.\,
    \begin{align}\label{eqn:assumption-quadratic-growth}
        \sup_{x \in \Z^d} \, \langle x \rangle^{-2} \Big(\sum_{y \sim x} \omega(x, y)\Big)
        < \infty. 
    \end{align}
\end{assumption}
For variable speed random walks on $\Z$ with symmetric and radial conductances $\omega$, there is indeed an equivalence between non-explosiveness and the condition
\begin{align}
    \sum_{n = 0}^{\infty} \frac{n}{\omega(n, n+1)} = \infty.
\end{align} 
This is a consequence of a general criterion for stochastic completeness, i.e.\, the absence of finite-time explosions, for weakly spherically-symmetric graphs, cf. \cite[Theorem 9.25]{keller2021graphs}. Hence, quadratically growing conductances seem to be the right power-law scale in one dimension. More generally, let us point out that there is a sufficient (but generally \textit{not} necessary) condition \cite[Theorem 14.11]{keller2021graphs} that the variable speed random walk is non-explosive if
    \begin{align}
        \int_0^\infty \, \frac{r}{(\log \vert B_{\rho^\omega}(x, r) \vert) \lor 1 } \, \dd r 
        = \infty
    \end{align} 
for the balls of the \textit{intrinsic metric} $\rho^\omega$, to be  defined in \eqref{equation:intrinsic_metric}, with respect to some fixed vertex \(x\). 
This integral is indeed infinite in the case of at most quadratically growing conductances. But since this criterion is only sufficient but not necessary it cannot be used to show we have finite-time explosions  for conductances that grow like \(\jap{x}^{2+\epsilon}\) for some $\epsilon>0$ in $d\geq2$.

Under the above growth assumption on the conductances we can make sure that the SDE \eqref{sde} is actually well posed and we are not making vacuous statements. 
\begin{theorem}[Well-posedness]\label{theorem:well-posed} 
   Suppose that the conductances $\omega$ satisfy Assumption \ref{assumption:at_most_quadratically_growing}. 
   Then the variable speed random walk is non-explosive and the SDE \eqref{sde} is well-posed for initial values $\phi_0 \in \mathcal{S}'$ in the following sense. For every initial condition $\phi_0\in \mathcal{S}'$ and every family $(B_t(x))_{x\in \mathcal{C}(\omega),\,t\geq0}$ of independent standard Brownian motions, there exists a unique adapted process
    $\phi\in C([0,\infty);\mathcal{S}')$ such that there is an $N=N(\phi_0,\omega) \in \N$ for which $\phi \in C([0,\infty); H_{-N})$ almost surely and 
    \begin{equation}
        \phi_t = \phi_0 - \int_0^t \Delta_\omega \phi_s \dd s + \sqrt{2}B_t, \quad 0\leq t < \infty,
    \end{equation}
    holds as an identity in $H_{-(N+2)}$. Moreover, for every $x\in \mathcal{C}(\omega)$ and $t\geq 0$ the solution is pointwise given by \eqref{sde-formal-solution}, where 
   \begin{equation}
       p^\omega_t(x,y) = \mathbf{P}_x^\omega[X_t = y], \quad x,y \in \mathcal{C}(\omega), t\geq 0
   \end{equation}
   and $(P_t^\omega)_{t\geq 0}$ is the semigroup of the variable speed random walk acting on $\mathcal{S}'$. 
\end{theorem}
In the uniformly elliptic case, the well-posedness of \eqref{sde} follows from the classical results for infinite systems with Lipschitz coefficients in \cite[Theorem  4.1]{shiga_infinite_1980}. But their assumption C-2, in particular the inequality $(4.8)$, is too restrictive for our purposes as it can only work for uniformly bounded conductances. 

These difficulties stem from the fact that, if the conductances \(\omega(x, y)\) are unbounded, one naively loses growth regularity in every infinitesimal time step of the SDE \eqref{sde}. However, the same apparent difficulty already appears for the heat equation \(\dd \phi_t = (-\Delta_\omega \phi_t )\dd t\), so that we can overcome this obstacle by first clarifying the situation there for at most quadratically growing conductances \(\omega(x, y)\) and at most polynomially growing initial conditions \(\phi_0 \in \mathcal{S}'\), see Section~\ref{section:heat-semigroup}.

\subsection{Stationary measures}
Having established well-posedness of the dynamics on $\mathcal{S}'$, we now turn to its stationary measures. By Theorem \(\ref{theorem:well-posed}\), the solution of \eqref{sde} defines a Markov process on $\mathcal{S}'$. We denote its Markov semigroup by $(T_t)_{t\geq0}$, i.e.\,
\[
T_tF(\phi) :=
 \mathbb E_\phi[F(\phi_t)],
\quad
\phi\in\mathcal{S}',
\]
for bounded and measurable $F:\mathcal{S}'\to\mathbb{R}$, where $\mathbb{E}_\phi$ denotes expectation for the solution of \eqref{sde} started from the initial condition $\phi_0=\phi$.
We begin with a short excursion to statistical mechanics to define the canonical candidates for the stationary measures of $(T_t)_{t\geq 0}$. 

\subsubsection{Gaussian random fields as Gibbs measures}
We are interested in continuous spin systems with local spin space $\R$ whose Hamiltonian in a finite volume $\Lambda \Subset \mathcal{C}(\omega)$ is given by 
\begin{align}\label{eq:definition-hamiltonian}
    \mathscr{H}_{\Lambda}^\omega(\phi) := \frac{1}{2}\sum_{\substack{\{i,j\}: i\sim j\\ \{i,j\}\cap\Lambda\neq\varnothing}}\omega(i,j)(\phi_i - \phi_j)^2, \quad \phi \in \R^{\mathcal{C}(\omega)}.  
\end{align}
Here we use the shorthand notation
$\Lambda \Subset \mathcal{C}(\omega)$ to signify that $\Lambda$ is a \textit{finite} subset. 
Note that the sum in \eqref{eq:definition-hamiltonian} is always finite, since we are only summing over finitely many pairs of sites. As usual, one now defines the finite volume Gibbs measure with boundary condition $\eta \in \R^{\mathcal{C}(\omega)}$ by 
\begin{align*}
    \mu^\omega_{\Lambda;\eta}(B) = \frac{1}{\mathbf{Z}^\eta_{\Lambda;\omega}}\int e^{-\mathscr{H}_\Lambda^\omega(\phi_\Lambda \eta_{\Lambda^c})}\mathbf{1}_B(\phi_\Lambda \eta_{\Lambda^c}) \prod_{x \in \Lambda}\dd\phi_x, \quad B \in \mathscr{F},
\end{align*} 

where $\mathscr{F}$ is the product $\sigma$-algebra on $\R^{\mathcal{C}(\omega)}$ and the normalisation constant $\mathbf{Z}^\eta_{\Lambda;\omega}$ is of course given by 
\begin{align*}
    \mathbf{Z}^\eta_{\Lambda;\omega} 
    =
    \int e^{-\mathscr{H}_\Lambda^\omega(\phi_\Lambda \eta_{\Lambda^c})} \prod_{x \in \Lambda}\dd\phi_x.  
\end{align*}

These finite volume Gibbs measures define a specification $\gamma^\omega = (\gamma_\Lambda^\omega)_{\Lambda \Subset \mathcal{C}(\omega)}$ with kernels given by 
\begin{align*}
    \gamma_\Lambda^\omega(\cdot \lvert \eta) := \mu^{\omega}_{\Lambda;\eta}(\cdot). 
\end{align*}
The set of Gibbs measures compatible with $\gamma^\omega$ and are supported on $\mathcal{S}'$ will be denoted by $\mathscr{G}(\gamma^\omega)$ and contains exactly those measures $\mu$ on $\R^{\mathcal{C}(\omega)}$ that satisfy the DLR equations, i.e.\, $\mu \in \mathscr{G}(\gamma^\omega)$ if and only if 
\begin{align*}
    \forall \Lambda \Subset \mathcal{C}(\omega) \ \forall B \in \mathscr{F}: \quad \mu(B \lvert \mathscr{F}_{\Lambda^c})(\phi) = \gamma^\omega_\Lambda(B \lvert \phi) \quad \text{for $\mu$-almost all $\phi$},
\end{align*}
where $\mathscr{F}_{\Delta}=\sigma(\phi_x:x\in\Delta)$ is the sub-$\sigma$-algebra generated by the spins inside the (possibly infinite) subset $\Delta \subset \mathcal{C}(\omega)$ and $\Lambda^c := \mathcal{C}(\omega) \setminus \Lambda$. 
In order to speak about properties and characterisations of stationary measures for the SDE \eqref{sde} but also about the existence of infinite-volume Gibbs measures for the specification $\gamma^\omega$, we first need to make sure that these actually exist. This is the purpose of the following assumption. 
\begin{assumption}[Quantitative transience]\label{assumption:green_function_is_tempered}
    The conductances \((\omega(x, y))_{x, y}\) are such that the random walk is transient and the map \[\mathcal{C}(\omega) \ni x \mapsto G^\omega(x, x) = \int_0^{\infty} p^\omega_t(x, x) \dd t\] is an element of \(\mathcal{S}'\). 
\end{assumption}

Usually, the description of the elements of the set of infinite-volume Gibbs measures is quite subtle. In our case, we can at least give a very explicit description of the extremal elements of $\mathscr{G}(\gamma^\omega)$. For this, let us first introduce some more notation. We say that a function $\eta : \mathcal{C}(\omega) \to \R$ is $\Delta_\omega$-harmonic (or just harmonic), if $\Delta_\omega \eta \equiv 0$. 

\begin{prop}\label{proposition:gaussian-gibbs-measures}
    Under Assumptions \ref{assumption:at_most_quadratically_growing} and \ref{assumption:green_function_is_tempered} the set $\mathscr{G}(\gamma^\omega)$ is non-empty. If $\mu$ is an extremal element of $\mathscr{G}(\gamma^\omega)$, then there exists a $\Delta_\omega$-harmonic function $\eta: \mathcal{C}(\omega) \to \R$ such that $\mu$ is Gaussian with mean $\eta$ and covariance matrix $G^\omega = (G^\omega(x,y))_{x,y\in \mathcal{C}(\omega)}$. 
    In other words, the characteristic function of $\mu$ is given by 
    \begin{align*}
        C_\eta: \mathcal{S} \ni u \mapsto \exp\left(i\langle u,\eta \rangle - \int_0^\infty \norm{P^\omega_s u}^2_2 \dd s\right),
    \end{align*} 
    where $\eta$ is a $\Delta_\omega$-harmonic function. We will denote this Gaussian measure $\mu$ by $\mathcal{N}(\eta, G^\omega)$. 
    Moreover, every element $\nu \in \mathscr{G}(\gamma^\omega)$ is of the form 
    \begin{align*}
        \nu(\cdot) = \int_{\mathcal{S}'}\mathcal{N}(\eta, G^\omega)(\cdot)\varrho(\dd\eta),
    \end{align*}
    where $\varrho$ is a probability measure on $\mathcal{S}'$ with 
    \begin{align*}
        \varrho\left(\eta \emph{ is $\Delta_\omega$-harmonic}\right) = 1. 
    \end{align*}
\end{prop}

The proof of this statement for the case where the conductances $\omega$ are constant can be found in \cite[Chapter~8]{friedli_statistical_2017} and for more general interactions in \cite[Theorem~13.24]{georgii_gibbs_2011}. In our setting the proofs work analogously. 

\subsubsection{Stationary iff reversible}

Later we will see that in the setting of quadratically growing conductances, wild things can happen. To show that all stationary measures are actually also reversible for the dynamics, we therefore need to impose a slightly more strict condition on the growth of the conductances. 

\begin{assumption}[Subquadratic growth]\label{assumption:subquadratic-growth}
    The conductances $(\omega(x,y))_{x,y \in \Z^d}$ are such that there is some \(\epsilon > 0\) with
    \begin{align}\label{assumption-eqn:at_most_subquadratically_growing}
        \sup_{x} \jap{x}^{-2 + \epsilon} \sum_{y \sim x} \omega(x, y)
        < \infty.
    \end{align}
\end{assumption}

Under these assumptions, we can now give the following Gibbsian characterisation of the stationary measures of the SDE~\eqref{sde}. Note that we do not require shift-invariance and that the results in particular apply in any dimension $d\geq 3$ if the conductances are uniformly elliptic. 

\begin{theorem}\label{theorem:stationary-iff-reversible}
    Assume that the conductances $\omega$ satisfy Assumption~\ref{assumption:green_function_is_tempered} and Assumption~\ref{assumption:subquadratic-growth}. Then for any probability measure $\nu$ on $\mathcal{S}'$ the following statements are equivalent. 
    \begin{enumerate}[i.]
        \item $\nu$ is a stationary measure for \eqref{sde}.
        \item $\nu$ is a reversible measure for \eqref{sde}. 
        \item $\nu$ is a Gibbs measure with respect to the specification $\gamma^\omega$. 
    \end{enumerate}
\end{theorem}
We will explain the proof strategy for this result in Section~\ref{section:proof-strategy}. Its relation to the existing literature on interacting particle systems is provided in Section~\ref{section:literature}.

\subsubsection{Non-tempered Green function}
One could now ask what happens if one drops one of the assumptions. Regarding Assumption~\ref{assumption:green_function_is_tempered} we provide the following answer. 
\begin{theorem}\label{theorem:nonexistence-stationary}
    If \(x \mapsto G^\omega(x, x)\) is not an \(\mathcal{S}'\)-function, then there are no measures $\mu$ that are stationary with respect to $(T_t)_{t\geq 0}$ and supported on $\mathcal{S}'$. 
\end{theorem}

In particular, in dimensions one and two, if the conductances are uniformly elliptic, then the SDE~\ref{sde} is well-posed but does not admit stationary measures supported on $\mathcal{S}'$. 

\subsubsection{Non-reversible stationary measures and periodicity}
Let us now turn to the Assumption~\ref{assumption:subquadratic-growth}. Here, one can construct examples of conductances $\omega$ that do satisfy Assumption~\ref{assumption:green_function_is_tempered} but not Assumption~\ref{assumption:subquadratic-growth} and for which the conclusion of Theorem~\ref{theorem:stationary-iff-reversible} fails, i.e.\, there are non-reversible stationary measures. 

\begin{theorem}[Non-reversible stationary measures and time-periodicity]\label{theorem:non-reversible-periodic}
    Consider the quadratically growing conductances $\omega$ on $\Z$ given by $\omega(n,n+1) = n^2 + 1$ for $n\in \Z$. Then the SDE~\eqref{sde} admits both reversible and non-reversible stationary measures on $\mathcal{S}'(\Z)$. Moreover, there is an initial distribution $\nu_0$ on $\mathcal{S}'(\Z)$ and a time-period $\tau>0$ such that if we initialise \eqref{sde} with $\phi_0 \sim \nu_0$, then $\phi_t \not\sim \nu_0$ for all $t\in(0,\tau)$ and $\phi_\tau \sim \nu_0$. 
\end{theorem}

The main idea here is the following. Even for constant conductances, the eigenvalue equation $\Delta_\omega h = i h$ has at least one  non-trivial solution $h:\Z^d \to \C$. 
This is apparently also well-known for the more general framework of so-called \emph{Jacobi operators}, which moreover include \emph{discrete Schrödinger operators}, see \cite{teschl2000jacobi} for a very detailed monograph treatment. In \(d = 1\) these formally read \[f \mapsto (n \mapsto a(n) f(n+1) + a(n-1)f(n-1) + b(n) f(n))\] for \(a, b \colon \Z \to \R\) on functions \(f \colon \Z \to \C\). Our Laplace operators then correspond to \(a(n) = -\omega(n, n+1)\) and \(b(n) = \omega(n, n-1) + \omega(n, n+1)\).
However, the solutions to $\Delta_\omega h = i h$ do typically grow at least (stretched) exponentially and are thus not  elements of $\mathcal{S}'(\Z^d)$. As it turns out, quadratic growth of the conductances is slow enough so that the variable speed random walk does not explode in finite time but still sufficiently fast so that (generalized) eigenfunctions for the eigenvalue $i$ can live in $\mathcal{S}'(\Z^d)$. 

Let us point out that the non-reversible stationary measure $\nu$ can be obtained by integrating along the time-periodic orbit. At the moment, we are not aware of any other non-reversible stationary measures that can be obtained in an essentially different way. 
The full construction is given in Section~\ref{section:non-reversible-stationary-measures}. 

Theorem~\ref{theorem:non-reversible-periodic} tells us that for general quadratically growing conductances, a lot can happen. During the proof, and also already in the roadmap laid out in Section~\ref{section:proof-strategy}, we will see that the time-periodic behaviour of the SDE is connected to non-trivial periodic orbits for the random walk semigroup $(P^\omega_t)_{t\geq 0}$. The following result gives a spectral characterisation of the occurrence of such unexpected behaviour, although under an additional moment assumption. 

\begin{theorem}[Spectral characterisation]\label{theorem:spectral-characterisation}
    Assume that the conductances $\omega$ satisfy Assumption~\ref{assumption:at_most_quadratically_growing}.
    Suppose there is a probability measure \(\varrho\) on \(\mathcal{S}'\) such that
    \begin{enumerate}[(1)]
        \item there is some \(T > 0\) with \(\varrho \circ (P^\omega_T)^{-1} = \varrho \) but \(\varrho\big(\{\xi \in \mathcal{S}' \,\colon\, P_T^\omega \xi = \xi \} \big) < 1\),
        \item the moment condition \begin{align*}
        \varrho\big[\Vert \xi \Vert_{H_{-N}}^2 \big]
        = \sum_{x} \langle x\rangle^{-2N} \varrho\big[\vert \xi(x)\vert^2 \big]
        < \infty
    \end{align*} holds for some \(N \geq 0\) .
    \end{enumerate}
    Then, \(\Delta_\omega\) has an eigenvalue in \(i \mathbb{R} \setminus \{0\}\) with eigenvector $\psi = \xi + i \eta$ such that $\xi,\eta \in H_{-N} \subseteq \mathcal{S}'$.
    The converse holds too and in fact \(\varrho\) can be chosen to be stationary with respect to \((P^\omega_t)_{t \geq 0}\) and with \[\varrho(\{\xi \in \mathcal{S}' \,\colon\, P^\omega_t \xi = \xi \, \forall t \geq 0\}) = 0.\]
\end{theorem}

Note that assumption (1) is in particular satisfied for \((P_t)_{t \geq 0}\)-invariant \(\varrho\) if \[\varrho\big(\{\xi \in \mathcal{S}' \,\colon\, P_t \xi = \xi \, \forall t \geq 0 \} \big) < 1.\] Indeed, because of continuity \(\{\xi:P_t\xi=\xi\ \forall t\geq 0\}=\bigcap_{q\in\mathbb{Q}_+}\{\xi:P_q\xi=\xi\}\), so if the left-hand side has mass strictly less than $1$, there must be a $T>0$ such that assumption (1) is satisfied. 

In the setting of Theorem~\ref{theorem:spectral-characterisation} and under the additional Assumption~\ref{assumption:green_function_is_tempered}, the measure $\varrho$ can then be used to construct non-reversible time-stationary measures $\nu$ for the SDE~\eqref{sde} via 
\begin{equation}
    \nu = \int_{\mathcal{S}'}\mathcal{N}(\xi, G^\omega) \varrho(\dd \xi),  
\end{equation}
see Proposition~\ref{proposition:gaussian-disintegration} and Proposition~\ref{lemma:gibbs-reversible} for details. The time-periodic orbits can be constructed in a similar fashion by using the eigenfunction $\psi = \xi + i\eta$ for the eigenvalue $i\theta$ and putting
\begin{equation}
    \nu_t = \mathcal{N}(P^\omega_t \xi, G^\omega), \quad t \geq 0. 
\end{equation}
Because of $P_t^\omega \xi = \cos(\theta t) \xi + \sin(\theta t)\eta$ this gives us a time-periodic orbit in $\mathcal{S}'$ for the semigroup $(P_t)_{t\geq 0}$ with period $2\pi/\abs{\theta}$ and thus $(\nu_t)_{t\geq0}$ is also a time-periodic orbit of laws for the SDE \eqref{sde}. 
\begin{question}
    At the moment, it is not clear to us if the moment condition (2) in Theorem~\ref{theorem:spectral-characterisation} can be dropped. See Remark~\ref{remark:support_on_H_minus_N}, Remark~\ref{remark:approximate_point_spectrum_of_Delta_on_H_minus_N} for more details.
\end{question}

\subsubsection{Random conductances}
Now that we have seen what can possibly go wrong when considering general quadratically growing conductances, one could ask what happens typically, i.e.\, for conductances sampled from some probability distribution. We show that under a certain moment condition, the disordered system almost surely reproduces the behaviour seen in Theorem~\ref{theorem:stationary-iff-reversible}. 

\begin{theorem}\label{theorem:random-conductances}
   Assume that the nearest-neighbour conductances $\omega$ are sampled from an ergodic translation-invariant probability measure $\mathbb{P}$ such that \begin{enumerate}[(1)]
       \item $\mathbb{P}$-almost every sample $\omega$ has exactly one infinite connected component, which we will denote by $\mathcal{C}(\omega)$,  and 
       \item the moment condition 
       \begin{equation}\label{moment-condition}
           \mathbb{E}[\omega(x,y)^{(d/2)+\delta}] < \infty, \quad x,y \in \Z^d,
       \end{equation}
       holds for some $\delta > 0$. 
   \end{enumerate} 
   Then the following statements hold for $\mathbb{P}$-almost every $\omega$. 
   \begin{enumerate}[i.]
       \item There is a well-defined semigroup $(T^\omega_t)_{t\geq 0}$ associated to the SDE \eqref{sde} with conductances $\omega$.
       \item All stationary measures for the SDE \eqref{sde} supported on $\mathcal{S}'(\mathcal{C}(\omega))$ are also reversible and in particular Gibbs measures for the specification $\gamma^\omega$. 
   \end{enumerate}
\end{theorem}

Let us recall from Theorem~\ref{theorem:nonexistence-stationary} that there exist \((T^\omega_t)_{t \geq 0}\)-stationary (and then also reversible) probability measures on \(\mathcal{S}'\) if and only if \(x \mapsto G^\omega(x, x)\) is an \(\mathcal{S}'\)-function (\(\mathbb{P}\)-almost surely). So if this is not the case, then $ii.$ in Theorem~\ref{theorem:random-conductances} is a vacuous statement as there are no stationary measures supported on $\mathcal{S}'$. \medskip 

In the case where the moment bound \eqref{moment-condition} only holds for $\delta=0$ we can at least rule out the occurrence of purely imaginary eigenvalues of the Laplacian \(\Delta_\omega\) on \(\mathcal{S}'(\mathcal{C}(\omega))\), but this is a weaker statement than showing that there are no non-trivial stationary orbits. 

\begin{theorem}\label{theorem:random-conductances-2}
    Assume that the nearest-neighbour conductances $\omega$ are sampled from an ergodic translation-invariant probability measure $\mathbb{P}$ such that \begin{enumerate}[(1)]
       \item $\mathbb{P}$-almost every sample $\omega$ has exactly one infinite connected component, which we will denote by $\mathcal{C}(\omega)$,  and 
       \item they satisfy the moment condition 
       \begin{equation}
           \mathbb{E}[\omega(x,y)^{d/2}] < \infty, \quad x,y \in \Z^d. 
       \end{equation}
   \end{enumerate} 
   Then the following statements hold for $\mathbb{P}$-almost every $\omega$. 
   \begin{enumerate}[i.]
       \item There is a well-defined semigroup $(T^\omega_t)_{t\geq 0}$ associated to the SDE \eqref{sde} with conductances $\omega$.
       \item For all $\lambda \in \mathbb{C}\setminus [0,\infty)$ there are no non-trivial solutions $\psi \in  \mathcal{S}'(\mathcal{C}(\omega))$ to the eigenvalue equation $\Delta_\omega \psi = \lambda \psi$ and in particular, there are no $(P^\omega_t)_{t\geq 0}$-invariant probability measures $\varrho$ on $\mathcal{S}'(\mathcal{C}(\omega))$ that satisfy both 
        \begin{equation*}
            \varrho\big(\{\xi \in \mathcal{S}' \,\colon\, P_T^\omega \xi = \xi \} \big) < 1
        \end{equation*}
       for some \(T > 0\) and the moment condition 
        \begin{equation*}
           \varrho\big[\Vert \xi \Vert_{H_{-N}}^2 \big]
        = \sum_{x} \langle x\rangle^{-2N} \varrho\big[\vert \xi(x)\vert^2 \big]
        < \infty
        \end{equation*}
        for some $N\geq 0$. 
     \end{enumerate}
\end{theorem}

In this paper, we have mainly been focussed on the case of deterministic conductances and have not tried to optimize the statement for random conductances too much. Therefore we do not believe that the moment condition \eqref{moment-condition} is optimal. In particular, in the case of i.i.d.~conductances, one does not need any moment bounds in order to get upper and lower bounds on the Green's function, as long as the associated percolation model is supercritical, see \cite{andres_invariance_2013}. It is very plausible that the equivalences in Theorem~\ref{theorem:stationary-iff-reversible} also hold in this setting, but we have not yet made attempts to verify this. 
Currently, we have the impression that the somewhat degenerate behaviour exhibited in Theorem \ref{theorem:non-reversible-periodic} requires a quite rigid structure of the conductances and cannot easily be reproduced by ergodic and translation-invariant disorder. \medskip 

Before we proceed to discuss how our results relate to the existing literature, let us provide an example of random environments that fit into the assumptions of Theorem \ref{theorem:random-conductances}. 

\begin{example}
    A somewhat canonical example for an ergodic, translation-invariant probability measure \(\mathbb{P}\) which can satisfy the properties from Theorem~\ref{theorem:random-conductances} is any \(\mathbb{P}\) that has conductances
    \begin{align}
        \omega(x, y) = c_p(x, y) c_w(x, y),
    \end{align} where \(c_p\) is sampled by a supercritical Bernoulli bond percolation measure with parameter \(p > p_c\) independently from \(c_w\), which is \emph{strictly positive} and sampled from an ergodic, translation-invariant probability measure.
    
    The conductances \(\omega\) now have more or less by definition \(\mathbb{P}\)-a.s.\ a unique infinite cluster \(\Cluster(\omega)\) for \(d \geq 2\) and the moment conditions translate directly to moment conditions on \(c_w\).

    The property that \(x \mapsto G^\omega(x, x)\) be an \(\mathcal{S}'\)-function \(\mathbb{P}\)-a.s.\ is a little bit more subtle. Let us additionally assume that \(d \geq 3\), so that simple random walk on \(\Z^d\) is transient, and \(\mathbb{E}[\frac{1}{\omega(x, y)} \mathbf{1}_{\omega(x, y) > 0}] \sim \mathbb{E}[\frac{1}{c_w(x, y)}] < \infty\), \(x, y \in \Z^d\), to be able to compare the behaviour w.r.t.\ transience of the random walk on \(\Cluster\) with the simple random walk.
    In this situation, it is not so hard to show, identifying the Green function as resistance to infinity and using Thomson's principle (cf. \cite[Theorem 2.11, Proposition 2.12]{LyonsPeres2016}) together with the result for pure Bernoulli percolation \cite[Proposition 6.2]{BarlowHambly2009}, that even \(\mathbb{E}[G^\omega(0, 0) \mathbf{1}_{0 \in \Cluster(\omega)}] < \infty\) and hence
    \begin{align*}
        \mathbb{E}\big[\Vert \sqrt{G^\omega}(\cdot, \cdot) \Vert_{H_{-N}}^2 \big]
        = \sum_{x \in \Z^d} \jap{x}^{-2N} \mathbb{E}[G^\omega(0, 0) \mathbf{1}_{0 \in \Cluster(\omega)}]
        < \infty
    \end{align*} for \(N = N(d) \geq 0\) big enough.

    Let us finally remark that very similar considerations can be made for long-range percolation models instead of Bernoulli bond percolation, using as input the impressive result \cite[Theorem 1.20]{Sapozhnikov2017}.
\end{example}

\section{Related literature}\label{section:literature}
\subsection{The critical Ornstein--Uhlenbeck process}
The critical Ornstein–Uhlenbeck process on the lattice \eqref{sde-intro} was studied in a series of works addressing its large-scale and long-time behaviour, including an invariance principle and large deviations for empirical means \cite{deuschel_1989}, followed by algebraic $L^2$-decay rates in \cite{deuschel_94}. It has also been studied in relation to aging phenomena, both from the perspective of statistical physics \cite{parisi} and mathematics \cite{dembo_deuschel}. The critical Ornstein--Uhlenbeck process is particularly well suited for studying such phenomena because of its linear and thus Gaussian structure. This makes it mathematically tractable while it nevertheless exhibits genuinely nontrivial out-of-equilibrium behaviour. 

\subsection{Stationary iff reversible}
After Holley's initial work \cite{holley_free_1971} classifying all shift-invariant stationary measures of stochastic Ising models, the follow-up works \cite{higuchi_results_1975,moulin_ollagnier_free_1977, kunsch_non_1984,jahnel_attractor_2019,jahnel_dynamical_2023} extended these results to quite general, even non-reversible, classes of interacting particle systems. Beyond shift-invariance, the groundbreaking work \cite{holley_one_1977} opened the door to a full classification of the stationary measures of reversible systems with compact local state space in dimensions one and two. These ideas were later extended to diffusive systems in \cite{holley_diffusions_1981}. Other notable works in this direction include \cite{bogachev_invariance_2004, fritz_infinite_1982,fritz_gradient_1987,fritz_stationary_1986, fritz_stationary_1998, fritz_reversibility_1997}. Note that most of these results either only work in dimensions one and two or under quite restrictive assumptions on the interaction or the local spin space, so in particular, they cannot be applied to the models we consider here. 

Quite recently, in \cite{lammers_non-reversible_2024} it was shown that a stochastic Ising model may admit non-reversible stationary measures on directed trees at  sufficiently large inverse temperatures and with suitable non-uniform couplings. The interesting aspect of this result is that the analogous processes on $\Z^d$, at least for $d=1,2$, do not have non-Gibbsian -- and thus non-reversible -- stationary measures.  To our knowledge, this was the first time that an example of a stochastic Ising model with a non-reversible stationary state has been constructed.
Our results show that a similar phenomenon can happen for the continuous counterpart of the Glauber dynamics, i.e.\, the Langevin dynamics for the Gaussian free field. However, in comparison to \cite{lammers_non-reversible_2024}, we believe that our construction is a bit more transparent and still lives on $\Z^d$, so it does not require a fine-tuned directed tree structure.  
\subsection{Time-periodic behaviour in infinite systems}
Apart from classifying the stationary measures, one major part of the literature on interacting particle systems deals with the actual study of their long-time behaviour, in particular the convergence to time-stationary measures and the question of ergodicity. For continuous-time Markov processes on finite state spaces and diffusions in finite-dimensional spaces, this is a significantly simpler question and rather well-understood. But
for interacting particle systems or diffusions indexed by $\Z^d$ this is much more subtle, not only because of the possible existence of multiple stationary measures but also due to the possibility of time-periodic orbits. 
Over the past years, a number of results have ruled out such time-periodic orbits under various assumptions. In particular, in dimensions one and two, their absence has been established for reversible interacting particle systems with finite local state spaces or under product measure assumptions \cite{jahnel_long-time_2025,koppl_absence_2026}, and more generally for broad classes of interacting particle systems in \cite{mountford_coupling_1995,ramirez_relative_1996,jahnel_restriction_2026}. 
The first two of these works rely on time-averaged variants of the relative entropy method of Holley and Stroock \cite{holley_one_1977}, whereas the others are based on quantitative comparisons with finite-volume systems that are only applicable in one spatial dimension. On the other hand, examples of interacting particle systems exhibiting time-periodic behaviour have recently been constructed in \cite{maes_rotating_2011,jahnel_class_2014,jahnel_time-periodic_2025}. The second and third examples, however, involve interactions of unbounded range, albeit with exponentially decaying strength, and are rather implicit due to the techniques employed in their construction. By contrast, our results provide a transparent and explicit construction of a nearest-neighbour system exhibiting stable time-periodic behaviour, without introducing an explicit breaking of time-reversal symmetry by a deterministic drift as in \cite{maes_rotating_2011}. This last point is somewhat unexpected, as time-periodic behaviour in physical systems typically requires some kind of non-reciprocity, see e.g.\ \cite{collet_rhythmic_2016,avni_nonreciprocal_2025}, or detailed balance to be broken in some other way, see e.g.\ \cite{niethammer_oscillations_2022, pego_temporal_2020}, but in our example this is not really the case. 

\section{Roadmap for the proofs}\label{section:proof-strategy}
As a first step, we start to derive characterisations of the stationary measures of the system of SDEs \eqref{sde}. This will essentially reduce the problem to studying the stationary measures of the semigroup $(P^\omega_t)_{t\geq 0}$ of the variable speed random walk with conductances $\omega$ acting on $\mathcal{S}'(\Z^d)$.

\begin{prop}\label{proposition:gaussian-disintegration}
    Suppose that the conductances $\omega$ satisfy Assumption~\ref{assumption:at_most_quadratically_growing} and Assumption~\ref{assumption:green_function_is_tempered}. 
    Let $\nu$ be a stationary measure for \eqref{sde}. Then 
    \begin{align}\label{eq:disintegration}
        \nu(\cdot) = \int_{\mathcal{S}'(\Z^d)}\mathcal{N}(\xi,G^\omega)(\cdot) \varrho(\dd\xi)
    \end{align}
    where the probability measure $\varrho(\cdot)$ on $\mathcal{S}'(\Z^d)$ satisfies 
    \begin{align}\label{eq:measure-preserving}
        \forall t \geq 0: \ \forall \text{ measurable } \Gamma \subset \mathcal{S}'(\Z^d): \quad \varrho(\xi \in \Gamma) = \varrho(\xi \circ P^\omega_t \in \Gamma).  
    \end{align}
    Conversely, any measure $\varrho(\cdot)$ which satisfies \eqref{eq:measure-preserving} defines a stationary measure $\nu$ for the SDE \eqref{sde} via \eqref{eq:disintegration}. 
\end{prop}

As a next step, we can then show that the set of reversible measures for the SDE \eqref{sde} and the set of Gibbs measures $\mathscr{G}(\gamma^\omega)$ coincide. 

\begin{prop}\label{lemma:gibbs-reversible}
    Suppose that the conductances $\omega$ satisfy Assumption~\ref{assumption:at_most_quadratically_growing} and Assumption~\ref{assumption:green_function_is_tempered}. 
    If $\nu \in \mathscr{G}(\gamma^\omega)$, i.e.\, if $\nu$ is of the form 
    \begin{align*}
        \nu(\cdot) 
        = 
        \int_{\mathcal{S}'(\Z^d)}\mathcal{N}(\xi,G^\omega)(\cdot)\varrho(\dd\xi)
    \end{align*}
    and the probability measure $\varrho(\cdot)$ satisfies 
    \begin{align}\label{eq:strict-invariance}
        \varrho\left(\xi: \ \forall t \geq 0: \  P^\omega_t \xi  = \xi \right) = 1,
    \end{align}
    then $\nu$ makes the process $(\phi_t)_{t\geq 0}$ defined by \eqref{sde} time-reversible, i.e.\, for all $0\leq s < t < \infty$ and all $u,v \in \mathcal{S}$ we have  
    \begin{equation}\label{reversible-sde}
        \begin{split}
        &\int_{\mathcal{S}'(\Z^d)}\mathbb{E}_\phi
        \left[
        \exp\left(i\left(\langle u, \phi_t\rangle + \langle v, \phi_s \rangle\right)\right)
        \right]
        \nu(\dd\phi)
        \\\
        =
        &\int_{\mathcal{S}'(\Z^d)}\mathbb{E}_\phi
        \left[
        \exp\left(i\left(\langle u, \phi_s\rangle + \langle v, \phi_t \rangle\right)\right)
        \right]
        \nu(\dd\phi).   
        \end{split}
    \end{equation}
    Conversely, every measure $\nu$ that satisfies \eqref{reversible-sde} is an element of $\mathscr{G}(\gamma^\omega)$. 
\end{prop}

If we compare \eqref{eq:measure-preserving} with \eqref{eq:strict-invariance}, then we see that in order to show that every stationary measure for the SDE \eqref{sde} is actually reversible, we need to put in some more work. In particular, we need to understand the analogous problem for the discrete heat semigroup $(P^\omega_t)_{t\geq 0}$. 
\medskip 

As seen above, we have reduced the proof of Theorem~\ref{theorem:stationary-iff-reversible} to ruling out the possibility of having \((P^\omega_t)_{t \geq 0}\)-stationary measures \(\varrho\) on \(\mathcal{S}'\) that are \textit{non-fixing}, i.e.\, they have the property that
\begin{align}\label{def:non-fixing}
    \varrho(\{\xi \in \mathcal{S}' \,\colon\,  \Delta_\omega \xi = 0 \}) 
    < 1.
\end{align}
This is precisely the content of the following result that together with the previous characterisations of reversible and stationary measures implies Theorem \ref{theorem:stationary-iff-reversible}. 

\begin{theorem}[Subquadratically growing weights]\label{theorem:stationary_measures_under_subquadratic_bounds}
    Assume there is some \(\epsilon > 0\) such that the subquadratic bound
    \begin{align}
        \sup_{x} \jap{x}^{-2 + \epsilon} \sum_{y \sim x} \omega(x, y)
        < \infty
    \end{align} holds.
    Then, every  \((P^\omega_t)_{t \geq 0}\)-stationary probability measure \(\varrho\) on \(\mathcal{S}'\) has the property
    \begin{align*}
        \varrho(\{\xi \in \mathcal{S}' \,\colon\,  \Delta_\omega \xi = 0 \}) 
        = \varrho(\{\xi \in \mathcal{S}' \,\colon\,  P^\omega_t \xi = \xi \,\forall t \geq 0 \}) 
        = 1.
    \end{align*}
\end{theorem}

The main idea behind the proof is the following reduction to a deterministic analytic statement. 

\begin{lemma}[Support in the harmonic functions if the heat semigroup is nice]\label{lemma:support_in_harmonic_functions_if_heat_semigroup_is_nice}
    Suppose the heat semigroup \((P^\omega_t)_{t \geq 0}\) has the property that for every \(\psi \in \mathcal{S}'\) there is an \(m\) which only depends on \(N := \min\{M \geq 0 \colon \psi \in H_{-M} \}\) and such that
\begin{align}\label{equation:pointwise_converge_to_zero_higher_order_derivative}
        (\Delta_\omega^m P^\omega_t \psi)(x)
        \xrightarrow[t \to \infty]{} 0, \quad \text{pointwise for all } x.
    \end{align} 
    Then,  every \((P^\omega_t)_{t \geq 0}\)-stationary probability measure \(\varrho\) on \(\mathcal{S}'\) satisfies
    \begin{align*}
        \varrho(\{\xi \in \mathcal{S}' \,\colon\,  \Delta_\omega \xi = 0 \}) 
        = \varrho(\{\xi \in \mathcal{S}' \,\colon\,  P^\omega_t \xi = \xi \,\forall t \geq 0 \}) 
        = 1.
    \end{align*}
\end{lemma}

Under the assumption of sub-quadratic growth, we can then verify this property of the heat semigroup $(P^\omega_t)_{t\geq0}$. 

\begin{prop}
    Assume that Assumption~\ref{assumption:subquadratic-growth} holds.
    Then, the heat semigroup \((P^\omega_t)_{t \geq 0}\) has the property that for all \(N \geq 0\) it holds 
    \begin{align}
        (\Delta^m_\omega P^\omega_t \psi)(x)
        \xrightarrow[t \to \infty]{} 0, \quad \text{pointwise for all } x,
    \end{align} for every \(\psi \in H_{-N}\) and every \(m > \tfrac{2}{\epsilon} N\).
\end{prop}

This is proved in Section~\ref{sec:unbounded-weights}, mostly by analytic means. 
In the case of uniformly bounded conductances, one can use a more probabilistic argument that also provides a good intuition for what is going on in the case of unbounded weights and where the difficulties of this more general case lie. For pedagogical reasons, we also provide this argument in Section~\ref{sec:bounded-weights}. 
\medskip 

If the conductances do not satisfy the subquadratic growth bound \eqref{assumption-eqn:at_most_subquadratically_growing} but only the weaker quadratic growth bound \eqref{eqn:assumption-quadratic-growth}, then things start to be a bit less nice. For general quadratically growing conductances, a lot can happen, as the conductances constructed in the proof of Theorem~\ref{theorem:non-reversible-periodic} in Section~\ref{section:non-reversible-stationary-measures} show. 

For dealing with the case of random conductances, we make use of the concept of an \textit{intrinsic metric} induced by the conductances. 

\begin{assumption}\label{assumption:intrinsic_metric}
    Let \(\pi^\omega(x) := \sum_{y \sim x} \omega(x, y)\) and the \emph{intrinsic metric} \(\rho^{\omega}\) be given by 
    \begin{align}
        l^{\omega}(x, y)
        := \frac{1}{\sqrt{1 + (\pi^\omega(x) \lor \pi^\omega(y))}}
    \end{align} for \(x, y \in \Cluster\), \(x \sim y\), and 
    \begin{align}\label{equation:intrinsic_metric}
        \rho^\omega(x, y)
        := \inf_{\gamma \colon x \to y} \sum_{\{z, z'\} \in \gamma} l^\omega(z, z')
    \end{align} where the infimum is over all nearest-neighbour paths in \(\Cluster\) between \(x\) and \(y\).
    Assume that 
    \begin{enumerate}
        \item the intrinsic metric has finite balls
            \begin{align}
                \vert B_{\rho^\omega}(x, r) \vert
                = \vert \{y \,\colon\, \rho^\omega(x, y) \leq r \} \vert
                < \infty
            \end{align} for x and all \(r \geq 0\).

        \item for some (and then equivalently for every) $o\in\mathcal C(\omega)$ and every $\psi\in\mathcal{S}'$,
        $$e^{-\alpha\rho^\omega(o,\cdot)}\psi\in\ell^2,\qquad \alpha>0.$$
    \end{enumerate}
\end{assumption}
In fact, (1) follows from (2). We state it here nevertheless because (1) is a condition in a theorem we cite. Also notice that it indeed is an \emph{intrinsic} metric because
\begin{align*}
    \sum_y \omega(x, y) \rho^\omega(x, y)^2
    \leq \sum_y \frac{\omega(x, y)}{1 + (\pi^{\omega}(x) \lor \pi^\omega(y))}
    \leq \frac{\pi^\omega(x)}{1 + \pi^{\omega}(x)} 
    \leq 1.
\end{align*}

This provides us with a sufficient criterion for the non-existence of non-trivial solutions $\psi \in \mathcal{S}'(\Z^d)$ to the eigenvalue equation $\Delta_\omega \psi = \lambda \psi$ for non-real $\lambda$. 

\begin{prop}[Shnol-type theorem]\label{proposition:shnol_type_theorem}
    Suppose that Assumption~\ref{assumption:intrinsic_metric} holds.
    Then, every \(\psi \in \mathcal{S}'\) with
    \begin{align}
        \Delta_\omega \psi = \lambda \psi
    \end{align} for some \(\lambda \in \mathbb{C} \setminus [0, \infty)\) is identically zero, i.e.\,  \(\psi = 0\).
\end{prop}

In particular, the set of conductances constructed to prove Theorem~\ref{theorem:non-reversible-periodic} in Section~\ref{section:non-reversible-stationary-measures} cannot give rise to an intrinsic metric that satisfies Assumption~\ref{assumption:intrinsic_metric}.

With Proposition~\ref{proposition:shnol_type_theorem} at hand, the proofs of Theorem~\ref{theorem:random-conductances} and Theorem~\ref{theorem:random-conductances-2} then mainly consist of checking that the conditions imposed on the law $\mathbb{P}$ ensure that the conductances satisfy Assumption~\ref{assumption:at_most_quadratically_growing} and Assumption~\ref{assumption:intrinsic_metric} or Assumption~\ref{assumption:subquadratic-growth}.

\subsection*{Notation}
From now on, we will mostly work with a fixed set of conductances $ (\omega(x,y))_{x,y\in \Z^d}$ satisfying the required assumptions and will drop the index $\omega$ from our notations, e.g.\, we will simply write $\Delta$ instead of $\Delta_\omega$ and $(P_t)_{t\geq0}$ instead of $(P_t^\omega)_{t\geq 0}$ and so on. 

\section{The heat semigroup and well-posedness}\label{section:heat-semigroup}
Before we start to deal with the SDE~\eqref{sde} and its stationary measures, we first make an analytic excursion to collect some useful estimates that will in particular allow us to show the well-posedness stated in Theorem~\ref{theorem:well-posed}. The properties established in this section will also help us for the subsequent analysis of the stationary measures. 

We will consider the heat semigroup \((P_t)_{t \geq 0}\) of the variable-speed random walk with formal generator \(-\Delta\) on various function spaces. In particular we will analyse its action on the following spaces which we already introduced at the beginning of Section~\ref{section:setting-results}. 
\begin{itemize}
    \item \(C_c\), the set of finitely supported functions on \(\Cluster\),
    
    \item \(\ell^2\), the space of square-summable functions on \(\Cluster\), with norm
    \begin{align*}
        \Vert u \Vert_{\ell^2}^2
        = \sum_x \vert u(x) \vert^2
    \end{align*}
    
    \item the function space \(H_N\), a weighted \(\ell^2\)-space, with norm
        \begin{align*}
            \Vert u \Vert_{H_{N}}^2
            = \sum_{x} \jap{x}^{2N} \vert u(x)\vert^2,
        \end{align*}
        
    \item the function space \(H_{-N}\), a weighted \(\ell^2\)-space, with norm
        \begin{align*}
            \Vert u \Vert_{H_{-N}}^2
            = \sum_{x} \jap{x}^{-2N} \vert u(x)\vert^2,
        \end{align*}

     \item the Schwartz function space \(\mathcal{S}\) given by
        \begin{align*}
            \mathcal{S}
            = \bigcap_{N \geq 0} H_{N},
        \end{align*}

    \item and the space of tempered functions \(\mathcal{S}'\) given by
        \begin{align*}
            \mathcal{S}'
            = \bigcup_{N \geq 0} H_{-N},
        \end{align*} which is the dual space of \(\mathcal{S}\).
\end{itemize}
Sometimes, we will also implicitly consider the complex-valued versions of these spaces.
We will now establish mapping properties of \((P_t)_{t \geq 0}\) on each of these spaces.

\subsection{Non-explosion and the heat semigroup on \(\ell^2\)}
We first note that the quadratic growth bound prevents finite-time explosions. 
\begin{lemma}[No explosion in finite time]\label{lemma:non-explosive} 
    Suppose that the conductances satisfy Assumption~\ref{assumption:at_most_quadratically_growing}. Then the variable speed random walk on $\Cluster$ is non-explosive. Equivalently, its transition kernel is conservative, i.e.\, 
    \begin{equation*}
        \mathbf{P}_x[\zeta = \infty] = 1 \quad \text{and} \quad \sum_{y\in \Cluster} p_t(x,y)=1
    \end{equation*}
    for every $x\in \Cluster$ and $t\geq 0$. 
\end{lemma}
This follows directly from the sufficient (but in general not necessary) criterion for non-explosiveness given in \cite[Theorem 14.11]{keller2021graphs}. We additionally provide an alternative proof in Section~\ref{subsection:H_N_norm_bounds}. 
\medskip 

Let us proceed by shortly discussing the classical case of the discrete heat semigroup on \(\ell^2\), where we can easily give reasonable self-adjoint versions of the formal Laplacian \(\Delta\). 

\begin{definition}[Dirichlet Laplacian, \(\ell^2\) semigroup]\label{definition:dirichlet_laplacian_l2_semigroup}
    Consider the Dirichlet form \(\mathcal{E}^{(D)}\) which arises as the closure of 
    \begin{align*}
        \mathcal{E}(f, g) = \frac{1}{2}\sum_{x, y} \omega(x, y) [f(x)-f(y)] [g(x) - g(y)]
    \end{align*} on \(C_c\) with respect to the norm \(\Vert\cdot\Vert_\mathcal{E} = (\mathcal{E}(\cdot, \cdot) + \Vert\cdot\Vert_{\ell^2}^2)^{1/2}\).
    Denote by \(\Delta^{(D)}\) the (self-adjoint) Dirichlet Laplacian, i.e.\, the unique self-adjoint operator \(A\) with 
    \begin{align*}
        \mathcal{E}^{(D)}(f, g)
        = \langle f, A g \rangle_{\ell^2}
    \end{align*} for all \(f \in \mathrm{Dom}(\mathcal{E}^{(D)})\), \( g \in \mathrm{Dom}(A)\).
    It generates the semigroup \((P_t)_{t \geq 0} := (\mathrm{e}^{-t \Delta^{(D)}})_{t \geq 0}\) on \(\ell^2\), cf. \cite[Section 1]{keller2021graphs}.
\end{definition}

We first collect some elementary properties of the Dirichlet Laplacian $\Delta^{(D)}$ and the associated semigroup on $\ell^2$. 

\begin{lemma}[The semigroup on \(\ell^2\)]\label{lemma:the_semigroup_on_l2}
    \begin{enumerate}[i.]
        \item  For \(f \in \ell^2\), we have the pointwise formula
        \begin{align*}
            (P_t f)(x)
            = (\mathrm{e}^{-t \Delta^{(D)}} f)(x)
            = \sum_{y} p_t(x, y) f(y).
        \end{align*}
        \item Under Assumption~\ref{assumption:at_most_quadratically_growing}, the same is true for the Neumann Laplacian \(\Delta^{(N)}\), which corresponds to the maximal form \(\mathcal{E}^{(N)}\) which is the restriction of \(\mathcal{E}\) to \(\{f \in \ell^2 \,\vert\, \mathcal{E}(f, f) < \infty\}\).
        \item The action of \(\Delta^{(D)}\), \(\Delta^{(N)}\), acting on their respective domains \(\mathrm{dom}(\Delta^{(D)})\), \(\mathrm{dom}(\Delta^{(N)})\), coincides with the formal Laplacian \(\Delta\), i.e.\,
        \begin{align*}
            \Delta^{(C)} f(x)
            = \Delta f(x), \quad f \in \mathrm{dom}(\Delta^{(C)}),
        \end{align*} for \(C \in \{D, N\}\).
    \end{enumerate}
\end{lemma}
\begin{proof}
    \textit{Ad i.:} This is e.g.\ \cite[Theorem 2.31]{keller2021graphs}.
    \newline 
    \textit{Ad ii.:} This is e.g.\ \cite[Corollary 5.6]{haeseler2012laplacians} together with \cite[Corollary 6.3 a)]{haeseler2012laplacians}, i.e.\, a consequence of the non-explosion in Lemma~\ref{lemma:non-explosive}. 
    \newline 
    \textit{Ad iii.:} This is e.g.\ \cite[Theorem 1.12]{keller2021graphs}.
\end{proof}

\subsection{The heat semigroup on \(\mathcal{S}\)}

As discussed in the previous section, the heat semigroup \((P_t)_{t \geq 0}\) has very good mapping properties on \(\ell^2\). For our purposes, we need to extend the action of \((P_t)_{t \geq 0}\) to the space of tempered functions \(\mathcal{S}'\). To do this, we will first study how \((P_t)_{t \geq 0}\) acts on \(\mathcal{S}\), and in particular how it \textit{quantitatively} acts on each \(H_{N}\), \(N \geq 0\).

\subsubsection{Schwartz-seminorm bounds}\label{subsection:H_N_norm_bounds}

Let us first discuss the following very important bound on the time-evolution of \(H_{N}\)-norms under \((P_t)_{t \geq 0}\). It will allow us to conclude that the variable-speed random walk semigroup \((P_t)_{t \geq 0}\) maps \(\mathcal{S}\) into \(\mathcal{S}\) continuously in the topology of \(\mathcal{S}\).
Observe that very naive bounds where we would pretend that the random walk moves at linear speed in the worst case direction cannot achieve this for unbounded weights and we would have to use to some extent the reversibility of the random walk to even see non-explosion if the inverses of the edge weights are summable.

\begin{lemma}[\(H_N\)-norm under the heat semigroup]\label{lemma:semigroup_norm_bound_schwartz}
    Let Assumption~\ref{assumption:at_most_quadratically_growing} hold, so that
    \begin{align}
        K_N := \sup_x \,\, \langle x \rangle^{-2N} \sum_{y \sim x} \omega(x, y) \vert \langle x\rangle^N - \langle y \rangle^N \vert^2
    \end{align} is finite for all \(N \geq 0\).
    Then, we have, for all \(N \geq 0\), the bound
    \begin{align}\label{eq:semigroup_norm_bound_schwartz}
        \Vert P_t u\Vert_{H_{N}}
        \leq \e^{K_N t / 2} \Vert u\Vert_{H_{N}}.
    \end{align}
    In particular we have, for all \(N \geq 0\), the pointwise bound
    \begin{align}\label{eq:pointwise_control_heat_kernel}
        \sum_{y} \langle y \rangle^{2N} p_t(x, y)^2
        \leq \e^{K_N t} \langle x \rangle^{2N}.
    \end{align}
\end{lemma}
\begin{proof}
    Let \(u \in H_{N}\). We will work in a finite volume \(\Lambda \Subset \Cluster\) first and then take a limit. Consider the heat semigroup \((P_t^\Lambda)_{t \geq 0}\) on \(\Lambda\) with killing at the boundary, i.e.\, \((P_t^\Lambda)_{t \geq 0}\) is the semigroup associated with the variable-speed random walk inside \(\Lambda\), killed on leaving \(\Lambda\), and has the generator \(-\Delta_\Lambda\) given by
    \begin{align*}
        (\Delta_\Lambda f)(x)
        = \Delta(i_\Lambda f)
    \end{align*} for \(f \colon \Lambda \to \mathbb{R}\) and with \(i_\Lambda f\) being the extension of \(f\) to a function on \(\mathcal{C}\) by setting it to zero outside \(\Lambda\). We will identify \(P_t^\Lambda u\) and \(i_\Lambda P_t^\Lambda u\) for this proof.

    By \cite[Lemma 1.21]{keller2021graphs} we have \(P_t^{\Lambda} u \xrightarrow[\Lambda \uparrow \Cluster]{} P_t u\) in \(\ell^2\), hence also pointwise, giving via Fatou's Lemma \(\Vert P_t u \Vert_{H_N}^2 \leq \sup_{\Lambda \Subset \C} \Vert P_t^{\Lambda} u \Vert_{H_N}^2\). We see that it suffices to uniformly upper-bound the finite-volume norms \( \Vert P_t^{\Lambda} u \Vert_{H_N}^2\) by \(\e^{K_N t } \Vert u\Vert_{H_{N}}^2\).

    For that, fix  \(\Lambda \Subset \Cluster\) and let
    \begin{align*}
        U(t) = \Vert P_t^{\Lambda} u \Vert_{H_N}^2
        = \sum_{x} \jap{x}^{2N} \vert P_t^{\Lambda} u (x)\vert^2, \quad t \geq 0.
    \end{align*}
    As we have a \(C_0\)-semigroup on our hands, it comes naturally to differentiate this function in time \(t\) and then try to bound the derivative.
    Computing the time derivative, we get
    \begin{align}\label{equation:time_derivative_of_local_H_N_norm}
        \frac{\mathrm{d}}{\mathrm{d} t} U(t)
        = \sum_{x} \jap{x}^{2N} [ -2 (\Delta_\Lambda P_t^{\Lambda} u)(x) \, (P_t^{\Lambda}u)(x) ] 
        = -2 \mathcal{E}(P_t^{\Lambda}u, \jap{\cdot}^{2N} P_t^{\Lambda} u)
    \end{align} by discrete Gauss-Green.
    We will control the r.h.s. by the Caccioppoli inequality \cite[Theorem 12.4]{keller2021graphs} which here reads
    \begin{align*}
        \mathcal{E}(f, g^2 f)
        \geq -\frac{1}{2} \sum_{x} f(x)^2 \sum_{y \sim x} \omega(x, y) \vert g(x) - g(y)\vert^2
    \end{align*} for all finitely supported \(f\) and all functions \(g\).
    Indeed, the r.h.s. of \eqref{equation:time_derivative_of_local_H_N_norm} is bounded by
    \begin{align*}
        -2\mathcal{E}(P_t^{\Lambda}u, \jap{\cdot}^{2N} P_t^{\Lambda} u)
        &\leq \sum_{x} \vert P_t^{\Lambda} u (x) \vert^2 \sum_{y \sim x} \omega(x, y) \vert \jap{x}^N - \jap{y}^N\vert^2 \\
        &\leq K_N U(t). 
    \end{align*}
    Hence, Gronwall's Lemma gives
    \begin{align*}
        \Vert P_t^{\Lambda} u \Vert_{H_N}^2 = U(t)
        \leq \e^{K_N t}  U(0) 
        =  \e^{K_N t}  \Vert u_{\vert\Lambda} \Vert_{H_N}^2 
        \leq \e^{K_N t}  \Vert u\Vert_{H_N}^2 
        , \quad t \geq 0.
    \end{align*}
\end{proof}

As a modest application of Lemma~\ref{lemma:semigroup_norm_bound_schwartz} we now give an alternative proof of Lemma~\ref{lemma:non-explosive}, which states that the VSRW is indeed non-explosive for at most quadratically growing weights.
\begin{proof}[Alternative proof of Lemma~\ref{lemma:non-explosive}]
    Suppose we don't know yet that \(\mathbf{P}_x[\zeta = \infty] = 1\), so that a priori \(p_t(x,y) = \mathbf{P}_x^\omega[X_t = y, \zeta > t]\). In this situation, Definition~\ref{definition:dirichlet_laplacian_l2_semigroup} and Lemma~\ref{lemma:the_semigroup_on_l2} (i) and (iii) still make sense, so that we do not argue in a circular fashion.
    For fixed \(x \in \Cluster\), we set
    \begin{align*}
        M(t)
        := \sum_{y \in \Cluster} p_t(x, y)
        = \sum_{y \in \Cluster} (P_t \mathbf{1}_x) (y).
    \end{align*}
    Then, \(M(0) = 1\) and
    \begin{align}\label{eq:derivative_sum_of_transition_probabilities}
        \frac{\dd}{\dd t} M(t)
        = \sum_{y \in \Cluster} (-\Delta P_t \mathbf{1}_x)(y)
        = 0,
    \end{align} and hence
    \begin{align*}
        \sum_{y \in \Cluster} p_t(x, y) = 1.
    \end{align*}
    We can justify both equalities in \eqref{eq:derivative_sum_of_transition_probabilities} by noting that \(\mathbf{1}_x \in \mathcal{S}\) and the facts that the convergence \(\lim_{h \to 0} \frac{P_{t+h} u - P_t u}{h} = -\Delta P_t u\) happens in the topology of \(\mathcal{S}\) by Lemma~\ref{lemma:the_semigroup_on_schwartz_space} and that  \(\sum_y \Delta u(y) = 0\) for \(u \in \mathcal{S}\) by interchangeability of the appearing sums. The time-derivative in \(\mathcal{S}\) was in turn justified using Lemma~\ref{lemma:semigroup_norm_bound_schwartz}.
\end{proof}

\subsubsection{Mapping properties on $H_N$}

\begin{lemma}[The semigroup on \(H_N\)]\label{lemma:the_semigroup_on_H_N}
    Let Assumption~\ref{assumption:at_most_quadratically_growing} hold and fix \(N \geq 0\).
    Then, 
    \begin{enumerate}[i.]
        \item \(P_t\) maps \(H_N\) continuously into \(H_N\). More precisely \eqref{eq:semigroup_norm_bound_schwartz} holds.

        \item \((P_t)_{t \geq 0}\) is a \(C_0\)-semigroup on \(H_N\).

        \item The action of the generator of \((P_t)_{t \geq 0}\) on its domain is given by the formal Laplacian \(-\Delta\).

        \item The formal Laplacian \(\Delta\) maps \(H_{N+2}\) continuously (bounded) into \(H_{N}\) for all \(N \geq 0\).

        \item It is \(H_{N+2} \subseteq \mathrm{dom}_{H_N}(\Delta)\), i.e.\,
        \begin{align*}
            \frac{\dd}{\dd t} P_t u
            = \lim_{h \downarrow 0} \frac{P_{t+h}u - P_t u}{h}
                = (-\Delta) P_t u
                = P_t (-\Delta) u
        \end{align*} with convergence in \(H_{N}\) for all \(u \in H_{N+2}\).
    \end{enumerate}
\end{lemma}
\begin{proof}
        \textit{Ad i.:} This is proven in Lemma~\ref{lemma:semigroup_norm_bound_schwartz}.
         \newline
        \textit{Ad ii.:} We only have to show the strong continuity, i.e.\, that the map \(t \mapsto P_t u\) is continuous with respect to the \(H_N\)-norm for \(u \in H_{N}\). Given \(\epsilon > 0\), there is \(u_n \in C_c\) such that \(\Vert u - u_n\Vert_{H_N} < \epsilon\). Then,
        \begin{align*}
            \Vert P_t u - u\Vert_{H_N}
            \leq \Vert P_t u_n - u_n\Vert_{H_N} + (1 + \Vert P_t \Vert_{H_N \to H_N}) \epsilon.
        \end{align*} Now we only have to first let \(t \downarrow 0\) and then \(\epsilon \downarrow 0\), because for \(v \in C_c\) it holds
        \begin{align*}
            \Vert P_t v - v\Vert_{H_N}^2
            = \sum_{x} \jap{x}^{2N} \vert (P_t v)(x) - v(x) \vert^2
            \xrightarrow[t \downarrow 0]{} 0. 
        \end{align*}
        Indeed, \(P_t v \xrightarrow[t \downarrow 0]{} v\) holds pointwise by the strong continuity of the \(\ell^2\)-semigroup and we additionally have the majorization
        \begin{align*}
            \vert P_t v(x) \vert^2
            \leq \Vert P_t v \Vert_{H_{M}}^2 \jap{x}^{-2M}
            \leq \e^{K_M t} \Vert v \Vert_{H_{M}}^2 \jap{x}^{-2M}.
        \end{align*}
        Therefore,
        \begin{align*}
            \jap{x}^{2N}\vert P_t v(x)-v(x)\vert^2 &\leq 2\e^{K_M}\lVert v\rVert_{H_M}^2\jap{x}^{-2(M-N)}+2\jap{x}^{2N}\vert v(x)\vert^2.
        \end{align*}
        Now choose $M>N+d/2$, then for $v\in C_c$
        \begin{align*}
            \sum_{x\in\Cluster}\left(\jap{x}^{-2(M-N)}+\jap{x}^{2N}\vert v(x)\vert^2\right)<\infty.
        \end{align*}
        Hence, dominated convergence gives
        \begin{align*}
            \lVert P_t v-v\rVert_{H_N}\xrightarrow[t\downarrow 0]{}0,
        \end{align*}
        as claimed above. 
         \newline
        \textit{Ad iii.:} Let \(A\) be the generator of \((P_t)_{t \geq 0}\) on \(H_N\) and \(u \in \mathrm{dom}(A)\). Then, \(\frac{P_t u - u}{t} \xrightarrow[t \downarrow 0]{} Au\) in \(H_N\), hence in \(\ell^2\) and \(A u = -\Delta^{(D)} u = -\Delta u\) by Lemma~\ref{lemma:the_semigroup_on_l2}.
         \newline
        \textit{Ad iv.:} This is immediate by the  Assumption~\ref{assumption:at_most_quadratically_growing} and the weights being nearest-neighbours.
         \newline 
        \textit{Ad v.:} Let us first argue for the pointwise limit, and then upgrade to \(H_N\)-convergence. Let \(u \in H_{N+2}\). We already know that \(C_c \subseteq \mathrm{dom}_{\ell^2}(\Delta^{(D)})\) and \(-\Delta^{(D)}\) is a closed operator as the generator of a \(C_0\)-semigroup. Hence it suffices to see that with \(u_n := u \mathbf{1}_{\jap{\cdot} \leq n}\) we have \(u_n \xrightarrow[n \to \infty]{} u\) in \(H_{N+2}\), hence also in \(\ell^2\), and also \(\Delta^{(D)} u_n = \Delta u_n \xrightarrow[n \to \infty]{} \Delta u\) in \(H_{N}\), hence in \(\ell^2\), by item $iv.$ It follows that \(u \in \mathrm{dom}_{\ell^2}(\Delta^{(D)})\), so in particular we get the pointwise identity
        \begin{align*}
            \frac{(P_t u)(x) - u(x)}{t}
            = \frac{1}{t} \int_0^t \, (-\Delta P_s u)(x) \, \dd s.
        \end{align*}
        Applying Minkowski's integral inequality, we get
        \begin{align*}
            &\Big\Vert \frac{P_t u - u}{t} - (-\Delta) u \Big\Vert_{H_{N}}
            = \Big\Vert -\frac{1}{t} \int_0^t \, [\Delta P_s u - \Delta u] \, \dd s\Big\Vert_{H_{N}}
            \leq \frac{1}{t} \int_0^t \, \Vert \Delta (P_s u - u) \Vert_{H_N} \, \dd s \\
            &\leq \Vert \Delta\Vert_{H_{N+2} \to H_N} \frac{1}{t} \int_0^t \, \Vert  P_s u - u \Vert_{H_{N+2}} \, \dd s 
            \xrightarrow[t \downarrow 0]{} 0,
        \end{align*} where the convergence to zero of the r.h.s. is by strong continuity.
\end{proof}

Let us now combine this to collect all the desired properties of the semigroup $(P_t)_{t\geq 0}$ acting on $\mathcal{S}$.

\begin{lemma}[The semigroup on \(\mathcal{S}\)]\label{lemma:the_semigroup_on_schwartz_space}
    Under Assumption~\ref{assumption:at_most_quadratically_growing}, we have
    \begin{enumerate}[i.]
        \item For all \(t \geq 0\), the maps \(P_t \colon \mathcal{S} \to \mathcal{S}\) are continuous in the topology of \(\mathcal{S}\).

        \item For all \(u \in \mathcal{S}\), the map \(t \mapsto P_t u\) is continuous with respect to the topology of \(\mathcal{S}\).

        \item For \(u \in \mathcal{S}\), it is
            \begin{align*}
                \frac{\dd}{\dd t} P_t u
                = \lim_{h \downarrow 0} \frac{P_{t+h}u - P_t u}{h}
                = (-\Delta) P_t u
                = P_t (-\Delta) u
            \end{align*} with convergence in \(\mathcal{S}\).
    \end{enumerate}
\end{lemma}
\begin{proof}
    All of the assertions follow from Lemma~\ref{lemma:the_semigroup_on_H_N} by recalling that convergence in \(\mathcal{S}\) is equivalent to convergence in all \(H_N\)-norms.
\end{proof}

\subsection{The heat semigroup on \(\mathcal{S}'\)}

The continuity of \(P_t \colon \mathcal{S} \to \mathcal{S}\) allows us to extend the semigroup to \(\mathcal{S}'\), a priori by duality, but we can also obtain a pointwise expression. The properties of the semigroup acting on $\mathcal{S}'(\Z^d)$ are collected in the following lemma. 
\begin{lemma}[The semigroup on \(\mathcal{S}'\)]\label{lemma:mapping-properties-tempered-distributions}
     Under Assumption~\ref{assumption:at_most_quadratically_growing}, we have:
    \begin{enumerate}[i.]
        \item For \(\xi\in \mathcal{S}'\), let \(P_t \xi\) be defined via \(\langle \cdot, \cdot\rangle_{\mathcal{S}', \mathcal{S}}\)-duality. Then
        \begin{align}
            (P_t \xi)(x)
            = \sum_y \xi(y) p_t(x, y)
        \end{align} pointwise, and the sum converges absolutely, with
         \begin{align}\label{eq:semigroup_pointwise_bound_tempered}
            \vert (P_t \xi)(x) \vert
            \leq \e^{K_N t /2} \langle x \rangle^{N} \Vert \xi \Vert_{H_{-N}}.
        \end{align} Even better,
        \begin{align}\label{eq:semigroup_norm_bound_tempered}
            \Vert P_t \xi \Vert_{H_{-N}}
            \leq e^{K_N t/2} \Vert \xi \Vert_{H_{-N}}
        \end{align} for all \(N \in \mathbb{N}_0\).
        
        \item For \(f \in \mathcal{S}'\), let \(g =\Delta f\) be defined via \(\langle \cdot, \cdot\rangle_{\mathcal{S}', \mathcal{S}}\)-duality.
        Then 
        \begin{align*}
            g(x) = (\Delta f)(x)
        \end{align*} pointwise, where the r.h.s. means formal Laplacian.

        \item We have, for \(\xi \in \mathcal{S}'\), \(\frac{P_t \xi - \xi}{t} \xrightarrow[t \downarrow 0]{} -\Delta \xi\) in \(\mathcal{S}'\), and in particular pointwise.
    \end{enumerate}
\end{lemma}
\begin{proof}
    \textit{Ad i.:} Let  \(\xi \in \mathcal{S}'\). First, \(\psi = \sum_y \xi(y) p_t(\cdot, y)\) indeed is a member of \(\mathcal{S}'\) by the bound \eqref{eq:semigroup_pointwise_bound_tempered}, which in turn follows directly from \eqref{eq:pointwise_control_heat_kernel} and the Cauchy--Schwarz inequality.
        Now, we have, for all \(u \in \mathcal{S}\), that
        \begin{align*}
            \langle P_t \xi, u \rangle_{\mathcal{S}', \mathcal{S}}
            &= \langle \xi, P_t u \rangle_{\mathcal{S}', \mathcal{S}}
            = \sum_{x} \xi(x) (P_t u)(x)
            = \sum_{x} \xi(x) \sum_{y} u(y) p_t(x, y) \\
            &= \sum_{y} u(y) \Big(\sum_{x} \xi(x) p_t(y, x)\Big)
            = \langle \psi, u \rangle_{\mathcal{S}', \mathcal{S}},
        \end{align*} where we could interchange the sums by \eqref{eq:pointwise_control_heat_kernel} and \(u \in \mathcal{S}\), \(\xi \in \mathcal{S}'\). Hence, \(P_t \xi = \psi = \sum_y \xi(y) p_t(\cdot,y)\).
        The bound \eqref{eq:semigroup_norm_bound_tempered} follows e.g.\ by duality \eqref{equation:H_minus_N_norm_by_duality} and the corresponding bound for Schwartz functions \eqref{eq:semigroup_norm_bound_schwartz}.
        \newline 
        \textit{Ad ii.:} Similar to $i.$, using Assumption~\ref{assumption:at_most_quadratically_growing} to interchange sums.
        \newline 
        \textit{Ad iii.:} The convergence in \(\mathcal{S}'\) follows by definition per duality and Lemma~\ref{lemma:the_semigroup_on_schwartz_space}, $iii.$. The pointwise convergence follows because \(\mathbf{1}_x \in \mathcal{S}\).
\end{proof}
\begin{remark}[Semigroup on \(H_{-N}\) and analyticity on all \(H_N\), \(N \in \Z\)]\label{remark:C0_semigroup_on_H_minus_N_and_analyticity}
    In fact, the semigroup \((P_t)_{t \geq 0}\) is also a \(C_0\)-semigroup on \(H_{-N}\) for \(N \geq 0\). This follows by duality from Lemma~\ref{lemma:the_semigroup_on_schwartz_space}.
    One can even prove that it is an analytic semigroup on all \(H_N\), \(N \in \Z\). This is not completely clear a priori since \(\Delta_{H_{N}}\), with  \(-\Delta_{H_{N}}\) the generator of \((P_t)_{t \geq 0}\)  on \(H_N\), is not non-negative self-adjoint as \(\Delta\) on \(\ell^2\) (or even normal) w.r.t.\ the inner product on \(H_N\).
    As this is not of great importance to us here, we only give a proof sketch (for \(N \in \Z\)):
    \begin{enumerate}
        \item Consider the conjugated semigroup \(\widetilde{P}_t = \jap{\cdot}^{N} P_t \jap{\cdot}^{-N}\) on \(\ell^2\), which essentially per definition has the generator \(-\widetilde{\Delta} = \jap{\cdot}^{N} (-\Delta_{H_N}) \jap{\cdot}^{-N}\) with \(\mathrm{dom}(\widetilde{\Delta}) = \jap{\cdot}^{N} (\mathrm{dom}_{H_N}(\Delta)) = \{u \in \ell^2 \,\colon\, \jap{\cdot}^{-N} u \in \mathrm{dom}_{H_N}(\Delta)\}\).
        
        \item One can show that the form \(\widetilde{\mathcal{E}}(f, g) = \mathcal{E}(\jap{\cdot}^{-N} f, \jap{\cdot}^{N} g) = \langle \jap{\cdot}^{N} \Delta \jap{\cdot}^{-N} f,  g \rangle_{\ell^2}\) on \(C_c\) has a continuous extension to \(\mathrm{dom}(\mathcal{E}^{(D)})\), which we will also denote by \(\widetilde{\mathcal{E}}\).

        \item Using elementary estimates and the assumption of at most quadratic growth \ref{assumption:at_most_quadratically_growing}, one can show that there is some \(c = c(N, K_N) > 0\)
        with
        \begin{align*}
            &\mathrm{Re}(\widetilde{\mathcal{E}}(f, f))
            \geq \mathcal{E}(f, f) - c \Vert f\Vert_{\ell^2}^2 \\
            &\abs{\mathrm{Im}(\widetilde{\mathcal{E}}(f, f)) } \leq (c/2)(\mathcal{E}(f, f) + \Vert f\Vert_{\ell^2}^2 ),
        \end{align*} and then further
        \(\vartheta = \vartheta(N, K_N) > 0\) such that the form \(\widetilde{\mathcal{E}}(\cdot, \cdot) + \vartheta \langle \cdot, \cdot\rangle_{\ell_2}\) is a closed sectorial form with an angle \(0 \leq \alpha = \alpha(N, K_N) < \frac{\pi}{2}\).
        
        \item By the representation theorem, there is a unique \(m\)-sectorial operator \(-A\) with domain \(\mathrm{dom}(A)\) associated with \(\widetilde{\mathcal{E}}\) that generates an analytic semigroup of angle \(\frac{\pi}{2} - \alpha\).
        
        \item Using resolvents, one can show that in fact \(A = \widetilde{\Delta} = \jap{\cdot}^{N} \Delta_{H_N} \jap{\cdot}^{-N}\).
        
         \item Conjugating back, \(P_t =  \jap{\cdot}^{-N} \widetilde{P}_t \jap{\cdot}^{N}\), we can transfer analyticity from \((\widetilde{P}_t)_{t \geq 0}\) on \(\ell^2\) to \((P_t)_{t \geq 0}\) on \(H_{N}\).
    \end{enumerate}
\end{remark}

\subsection{The stochastic convolution and well-posedness}

The results of the previous sections already tell us that the first term in \eqref{sde-formal-solution} is well-defined. 
Let us show now that also the noise term in \eqref{sde-formal-solution} is well-defined and actually behaves quite nicely. We sketch the main arguments needed for the proofs. Putting these two ingredients together yields Theorem~\ref{theorem:well-posed}. 

\begin{lemma}[Well-defined noise-term]\label{lemma:construction_of_nice_noise_term}
    Suppose \((\Omega, \mathcal{F}, \mathbb{P})\) is any probability space supporting independent Brownian motions \((B_t(x))_{x \in \Cluster, t \geq 0}\).
    The noise-term
    \begin{align*}
        X_t(x) := \sqrt{2}\sum_y \int_0^t \, p_{t-s}(x, y) \, \dd B_s(y)
    \end{align*} is well-defined, in particular, there is an \(N \geq 0\) such that there is a version of \(X\) as random continuous function
    \begin{align*}
        [0, \infty) \to H_{-N} \subseteq \mathcal{S}'.
    \end{align*}
    This version has the property that almost surely for all \(t \geq 0\) it holds that
    \begin{align}\label{equation:SDE_with_initial_config_0_integral_form}
       X_t = \int_0^t \, (-\Delta X_s) \, \dd s + \sqrt{2} B_t,
    \end{align} where the integral is an integral in \(H_{-(N+2)}\).
\end{lemma}
\begin{proof}
    We will heavily capitalize on the following bound for \(0 \leq t \leq T\),
    \begin{equation}\label{equation:pointwise_bound_laplace_of_heat_kernel}
    \begin{split}
        \big\vert \Delta p_{t}(\cdot, y)(x) \big\vert^2
        &= \big\vert \Delta p_{t}(x, \cdot)(y) \big\vert^2
        \leq \jap{y}^{-2M} \Vert \Delta P_t \mathbf{1}_x \Vert_{H_{M}}^2
        \lesssim_T \jap{y}^{-2M} \Vert P_t \mathbf{1}_x \Vert_{H_{M + 2}}^2 \\
        &\lesssim_T \jap{y}^{-2M} \Vert \mathbf{1}_x \Vert_{H_{M + 2}}^2
        = \jap{y}^{-2M} \jap{x}^{2(M+2)}, 
        \end{split}
    \end{equation} 
    so that the remaining work simply consists in noticing that this deterministic decay transfers to corresponding stochastic integrals. 
    To make use of it, we choose \(M \geq 0\) big enough s.t. \(\sum_x \jap{x}^{-M} < \infty\) and then choose \(N > M\) big enough s.t. \(\sum_x \jap{x}^{2(M+2) - 2N} < \infty\).
    
    Now, for \(x, y \in \Cluster\) (which are countably many) choose continuous in \(t\) versions of
    \begin{align*}
        X_t(x, y)
        := \sqrt{2}\int_0^t \, p_{t-s}(x, y) \, \dd B_s(y)
    \end{align*} 
    such that stochastic Fubini gives almost surely for all \(t \geq 0\) that
    \begin{align}\label{equation:stochastic_fubini_applied_to_X_t_x_y}
            X_t(x, y)
            = \sqrt{2}B_t(x) \mathbf{1}_{y = x} - \sqrt{2}\int_0^t \Big(\int_0^r \Delta p_{r-s}(\cdot, y)(x) \,\dd B_s(y) \Big) \dd  r,
    \end{align} with continuous in \(t\) versions on the r.h.s.
    
    Let us first show that almost surely, \(X_t(x) := \sum_y X_t(x, y)\) is well-defined for all \(t \geq 0\).
    We get, for fixed \(T \geq 0\), 
    \begin{equation}\label{equation:expectation_bound_compact_time_sup_X_t_x}
    \begin{split}
        &\mathbb{E}\big[\sup_{0 \leq t \leq T} \vert X_t(x, y)\vert^2 \big] \\
        &\lesssim \mathbb{E}\big[\sup_{0 \leq t \leq T} \vert B_t(x)\vert^2\big] \mathbf{1}_{y = x} + \mathbb{E}\Big[\sup_{0 \leq t \leq T} \Big\vert \int_0^t \Big(\int_0^r \Delta p_{r-s}(\cdot, y)(x) \,\dd B_s(y) \Big) \dd  r \Big\vert^2\Big] \\
        &\lesssim  T \mathbf{1}_{y = x} + \mathbb{E}\Big[\Big(\int_0^T \Big\vert \int_0^r \Delta p_{r-s}(\cdot, y)(x) \,\dd B_s(y) \Big\vert \dd  r \Big)^2\Big] \\
        &\lesssim  T\mathbf{1}_{y = x} +  \Big(\int_0^T \mathbb{E}\Big[\Big\vert \int_0^r \Delta p_{r-s}(\cdot, y)(x) \,\dd B_s(y) \Big\vert^2\Big]^{1/2} \dd r \Big)^2 \\
        &= T \mathbf{1}_{y = x}  + \Big(\int_0^T \Big[\int_0^r \big\vert \Delta p_{r-s}(\cdot, y)(x) \big\vert^2 \,\dd s \Big]^{1/2} \dd r \Big)^2  \\
        &\lesssim_T \mathbf{1}_{y = x} + \jap{y}^{-2M} \jap{x}^{2(M+2)} 
        \end{split}
    \end{equation} by Minkowski's integral inequality.

    Well-definedness of \(X_t(x) = \sum_y X_t(x, y)\) follows now by
    \begin{equation}\label{equation:expectation_bound_compact_time_sup_X_t_x-2}
    \begin{split}
        &\mathbb{E}\Big[\sup_{0 \leq t \leq T} \sum_y \vert X_t(x, y) \vert \Big]
        \leq \sum_y \mathbb{E}\Big[\sup_{0 \leq t \leq T} \vert X_t(x, y) \vert \Big]
        \leq \sum_y \mathbb{E}\big[\sup_{0 \leq t \leq T} \vert X_t(x, y) \vert^2 \big]^{1/2} \\
        &\lesssim_{T, x} \sum_y \jap{y}^{-M} 
        < \infty.
        \end{split}
    \end{equation}

    As a next step, we see that almost surely \(t \mapsto X_t(x)\) is continuous for all \(x \in \Cluster\). Indeed, this follows readily from continuity of \(t \mapsto X_t(x, y)\) and majorization provided by \eqref{equation:expectation_bound_compact_time_sup_X_t_x}.
    We can now derive the continuity \(t \mapsto X_t\) in \(H_{-N}\) from that, using again a majorization given by, and shown very similarly to \eqref{equation:expectation_bound_compact_time_sup_X_t_x},
    \begin{equation}\begin{split}
        \mathbb{E}\Big[ \Big\Vert\sup_{0 \leq t \leq T} \vert X_t\vert \Big\Vert_{H_{-N}}^2 \Big]
        &= \sum_{x} \jap{x}^{-2 N} \mathbb{E}\big[\sup_{0 \leq t \leq T} \vert X_t(x) \vert^2 ]
        \\
        &\lesssim_{T, M} \sum_{x} \jap{x}^{-2 N} \jap{x}^{2(M+2)}
        < \infty.
        \end{split}
    \end{equation}
    The identity \eqref{equation:SDE_with_initial_config_0_integral_form} follows now pointwise, i.e.\, for \(X_t(x)\), essentially from \eqref{equation:stochastic_fubini_applied_to_X_t_x_y}. Then we identify the integral as an integral in \(H_{-(N+2)}\) by the continuity of the integrand in \(H_{-(N+2)}\).
\end{proof}

\begin{lemma}[Uniqueness]\label{lemma:uniqueness-stochastic-part}
    Suppose \((\Omega, \mathcal{F}, \mathbb{P})\) is any probability space supporting independent Brownian motions \((B_t(x))_{x \in \Cluster}\) and, for some \(M \geq 0\), a random continuous function \(\psi \colon [0, \infty) \to H_{-M}, t \mapsto (\psi_t(x))_{x \in \Cluster}\) satisfying
    \begin{align*}
        \begin{cases}
            \psi_0 = \varphi_0, \\
            \dd \psi_t = (-\Delta \psi_t) \dd t + \sqrt{2} \dd B_t.
        \end{cases}
    \end{align*} 
    Then
    \begin{align*}
        \psi
        = P_{\cdot} \varphi_0 + X,
    \end{align*} where \(X\) is given by Lemma~\ref{lemma:construction_of_nice_noise_term}.
\end{lemma}
\begin{proof}
    Let \(Z_t := \psi_t - X_t\).
    Then, almost surely, for all \(t \geq 0\) it holds
    \begin{align}
        Z_t
        = \varphi_0 + \int_0^t \, (-\Delta Z_s) \, \dd s,
    \end{align} where the integral is an integral in \(H_{-M-2}\) by continuity of \(s \mapsto \Delta Z_s\) in \(H_{-M-2}\).
    This already implies \(Z_t = P_t \varphi_0\).
\end{proof}

\subsection{Well-posedness of the SDE}
We can now briefly put together the previous results to complete the proof of Theorem~\ref{theorem:well-posed}.

\begin{proof}[Proof of Theorem~\ref{theorem:well-posed}]
    The non-explosion of the variable speed random walk under Assumption~\ref{assumption:at_most_quadratically_growing} follows from Lemma~\ref{lemma:non-explosive}. In particular, its transition kernel $p_t(\cdot, \cdot)$ is conservative and the corresponding heat semigroup $(P_t)_{t\geq 0}$ is well-defined.

    Now fix an initial condition $\phi_0 \in \mathcal{S}'$. Then by the characterisation of $\mathcal{S}'$ there is some $M_0 = M_0(\phi) \in \N$ such that $\phi_0 \in H_{-M_0}$.
    By Lemma~\ref{lemma:mapping-properties-tempered-distributions} and Remark~\ref{remark:C0_semigroup_on_H_minus_N_and_analyticity} $(P_t)_{t\geq0}$ acts as a $C_0$-semigroup on $H_{-M_0}$, so the map $[0,\infty)\ni t \mapsto P_t \phi_0 \in H_{-M_0}$ is continuous. Moreover, Lemma~\ref{lemma:mapping-properties-tempered-distributions} gives the pointwise representation 
    \begin{equation*}
        (P_t\phi_0)(x) 
        =
        \sum_{y\in \Cluster}p_t(x,y)\phi_0(y), \quad x \in \Cluster,
    \end{equation*}
    where the sum converges absolutely. 

    Assumption~\ref{assumption:at_most_quadratically_growing} ensures that the formal Laplacian extends to a continuous map $\Delta: H_{-M} \to H_{-(M+2)}$ for any $M\geq 0$, so that also the map $[0,\infty) \mapsto \Delta P_s \phi_0 \in H_{-(M_0+2)}$ is continuous and can thus be integrated to yield 
    \begin{equation*}
        P_t\phi_0 = \phi_0 + \int_0^t (-\Delta)P_s\phi_0 \dd s
    \end{equation*}
    as an identity in $H_{-(M_0+2)}$. 

    On the other hand, Lemma~\ref{lemma:construction_of_nice_noise_term} yields some $M_1 \in \N$ and a version of the stochastic term 
    \begin{equation*}
        X_t(x) = \sqrt{2}\sum_{y \in \Cluster}\int_0^t p_{t-s}(x,y)\dd B_s(y), \quad x \in \Cluster,  
    \end{equation*}
    such that $X\in C([0,\infty); H_{-M_1})$ and with the property that almost surely for all $t\geq 0$ we have 
    \begin{equation}
        X_t = \int_0^t (-\Delta X_s) \dd s + \sqrt{2}B_t
    \end{equation}
    as an identity in $H_{-(M_1+2)}$. Now we can set $N:=M_0 \vee M_1$ and define $\phi_t := P_t \phi_0 + X_t$. By definition of $N$, both summands have continuous $H_{-N}$-valued paths and therefore almost surely $\phi \in C([0,\infty); H_{-N}) \subset C([0,\infty); \mathcal{S}')$. 

    The process is adapted since $P_t\phi_0$ is deterministic and the stochastic convolution $X_t$ is measurable with respect to the Brownian filtration up to time $t$. Additionally, we can add the two integral representations of the deterministic and the stochastic term to see that 
    \begin{equation*}
        \phi_t = \phi_0 - \int_0^t \Delta \phi_s \dd s + \sqrt{2}B_t, \quad t\geq 0,
    \end{equation*}
    holds as an identity in $H_{-(N+2)}$. Thus, $\phi$ is a solution of the SDE \eqref{sde} and by construction we have the pointwise representation 
    \begin{equation*}
        \phi_t(x) = (P_t\phi_0)(x) + \sqrt{2}\sum_{y\in \Cluster}\int_0^tp_{t-s}(x,y)\dd B_s(y),
        \quad x \in \Cluster, t \geq 0, 
    \end{equation*}
    which is precisely \eqref{sde-formal-solution}. 
    Uniqueness among all adapted processes with continuous paths in $H_{-M}$ for some $M \geq 0$ follows from Lemma~\ref{lemma:uniqueness-stochastic-part}. 
\end{proof}

\section{Characterisation of stationary and reversible measures}\label{sec:characterisation-stationary-reversible}
As a next step, we now start to derive characterisations of the stationary measures of the system of SDEs \eqref{sde}. For this, we adapt the classical strategy from \cite{holley_generalized_1978} to our setting. 
One of the main technical ingredients is the following simple observation that will help us simplify many computations. 

\begin{lemma}\label{lemma:martingale-property}
    Let $(\phi_t)_{t\geq0}$ be a solution of \eqref{sde}, $T>0$, and $u\in \mathcal{S}$. Then the process defined by 
    \begin{align*}
        Y^T_u(t) := \exp\left(i \langle P_{T-t}u, \phi_t\rangle + \int_0^{T \wedge t} \norm{P_{T-s}u}^2_2 \dd s\right), \quad 0 \leq t \leq T,
    \end{align*}
    is a martingale with respect to the filtration $\mathbb{F}$. 
\end{lemma}

\begin{proof}
    By Novikov's criterion, see e.g.\ \cite[Theorem 5.23]{le_gall_brownian_2016},  the stochastic exponential 
    \begin{align*}
        X_u^s := \exp\left(i\left(\langle u, \phi_{t \vee s}\rangle-\langle u, \phi_s\rangle + \int_s^{t\vee s}\langle \Delta u, \phi_r\rangle \dd r \right)+\norm{u}^2_2(t \vee s - s)\right)
    \end{align*}
    is a martingale. Now we can use the continuity of the map
    \begin{align*}
        [0, \infty)^2 \ni (s,t) \mapsto \langle P_t u , \phi_s \rangle \in \R,
    \end{align*}
    which is a consequence of Theorem~\ref{theorem:well-posed} and Lemma~\ref{lemma:mapping-properties-tempered-distributions}, 
    to see that  for any $0 \leq t_1 < t_2 \leq T$  the following limit exists almost surely
    \begin{align}\label{proof:eqn-limit}
        \frac{Y^T_u(t_2)}{Y^T_u(t_1)}
        =
        \lim_{n \to \infty}\prod_{k=0}^{n-1}X^{\tau_{n,k}}_{P_{T-\tau_{n,k}}u}(\tau_{n,k+1}),
    \end{align}
    where we use the equidistant partitions $\tau_{n,k} = t_1 + \frac{k}{n}(t_2 - t_1)$. By dominated convergence this convergence also holds in $L^1$. 
    The martingale property of the stochastic exponential $X$ and the tower rule for conditional expectations implies that for any event $H \in \calF_{t_1}^B$ and $1 \leq m \leq n$ we have
    \begin{align}\label{proof:eqn-iterate}
        \mathbb{E}\left[\prod_{k=0}^{m-1}X^{\tau_{n,k}}_{P_{T-\tau_{n,k}}u}(\tau_{n,k+1})\mathbf{1}_H \right]
        =
        \mathbb{E}\left[\prod_{k=0}^{m-2}X^{\tau_{n,k}}_{P_{T-\tau_{n,k}}u}(\tau_{n,k+1})\mathbf{1}_H \right],
    \end{align}
   with the convention that $\prod^{-1}_{k=0} = 1$. By using  that the convergence in \eqref{proof:eqn-limit} holds in $L^1$, and by iteratively using \eqref{proof:eqn-iterate} 
   \begin{align*}
       \mathbb{E}\left[\frac{Y^T_u(t_2)}{Y_u^T(t_1)}\mathbf{1}_H\right]
       =
       \lim_{n \to \infty}
       \mathbb{E}\left[\prod_{k=0}^{n-1}X^{\tau_{n,k}}_{P_{T-\tau_{n,k}}u}(\tau_{n,k+1})\mathbf{1}_H\right]
       =
       \mathbb{E}\left[\mathbf{1}_H\right].
   \end{align*}
   So by a monotone class argument, we see that $Y^T_u$ is indeed a martingale. 
\end{proof}

With this technical helper in place, we can now use it to derive 
the previously claimed representation for all stationary measures.

\begin{proof}[Proof of Proposition~\ref{proposition:gaussian-disintegration} ]
    We will use explicit calculations for the characteristic functional of $\nu$. By the assumed stationarity of $\nu$ we have for all $t\geq 0$ and $u\in \mathcal{S}$
    \begin{align*}
        C_\nu(u) 
        \overset{\text{def}}{=}
        \int_{\mathcal{S}'(\Z^d)}\exp(i\langle u, \phi\rangle)\nu(d\phi)
        &\overset{\text{stationary}}{=}
        \int_{\mathcal{S}'(\Z^d)}\mathbb{E}_\phi[\exp(i\langle u, \phi_t\rangle)]\nu(d\phi).
    \end{align*}
    We can rewrite this in terms of the process $(Y^t_u(s))_{s \geq 0}$ from Lemma \ref{lemma:martingale-property} and use its martingale property to see that 
    \begin{equation}\label{proof:stationary-characterisation-1}
    \begin{split}
        C_\nu(u)
        &=
        \int_{\mathcal{S}'(\Z^d)}\mathbb{E}_\phi[\exp(i\langle u, \phi_t\rangle)]\nu(d\phi) \\
        &=
        \exp\left(-\int_0^t \norm{P_{t-s}u}^2_2 ds\right)
        \int_{\mathcal{S}'(\Z^d)}\mathbb{E}_\phi[Y^t_u(t)]\nu(d\phi)
        \\
        &\overset{\ref{lemma:martingale-property}}{=}
        \exp\left(-\int_0^t \norm{P_{t-s}u}^2_2 ds\right)
        \int_{\mathcal{S}'(\Z^d)}e^{i\langle P_t u, \phi \rangle}\nu(d\phi). 
    \end{split}
    \end{equation}
    Let us denote the measure with characteristic function given by the last integral in \eqref{proof:stationary-characterisation-1} by $\nu_t$, i.e.\, 
    \begin{align*}
        C_{\nu_t}(u) 
        = \int_{\mathcal{S}'(\Z^d)}e^{i \langle u, \phi \rangle }\nu_t(d\phi) 
        =\int_{\mathcal{S}'(\Z^d)}e^{i\langle P_t u, \phi \rangle}\nu(d\phi). 
    \end{align*}
    Now send $t\to \infty$ and use the continuity of $u\mapsto \int_0^\infty \norm{P_t u}^2_2 \dd t$ which holds by Assumption~\ref{assumption:green_function_is_tempered}, to see that $(\nu_t)_{t \geq 0}$ converges to some probability measure $\varrho$ on $\mathcal{S}'(\Z^d)$ with characteristic function given by 
    \begin{align*}
        C_\varrho(u) = \exp\left(\int_0^\infty\norm{P_s u}^2_2\dd s \right)C_\nu(u). 
    \end{align*}
    But this implies that 
    \begin{align*}
        C_\nu(u) 
        &= \int_{\mathcal{S}'(\Z^d)}\exp\left(i\langle u, \xi\rangle - \int_0^\infty \norm{P_s u}^2_2 ds\right)\varrho(d\xi)
        \\\
    &=\int_{\mathcal{S}'(\Z^d)}\int_{\mathcal{S}'(\Z^d)}\exp(i\langle u, \phi\rangle)\mu_\xi(d\phi)\varrho(d \xi),
    \end{align*}
    so we have shown \eqref{eq:disintegration}. To show that the measure $\varrho$ satisfies \eqref{eq:measure-preserving}, notice that for any $t\geq 0$ we have by $\nu_s \rightharpoonup \varrho$ that 
    \begin{align*}
        \int_{\mathcal{S}'(\Z^d)}\exp(i \langle P_t u, \xi\rangle) \varrho(d\xi)
        &=
        \lim_{s \to \infty}\int_{\mathcal{S}'(\Z^d)}\exp(i\langle P_tu,\xi\rangle)\nu_s(d\xi)
        \\
        &=
        \lim_{s \to \infty}\int_{\mathcal{S}'(\Z^d)}\exp(i\langle u, \xi\rangle)\nu_{t+s}(d\xi)
        \\
        &=
        \int_{\mathcal{S}'(\Z^d)}\exp(i\langle u, \xi \rangle) \varrho(d\xi). 
    \end{align*}
    Thus, the measure $\varrho(\cdot)$ is indeed invariant with respect to the action of the semigroup $(P_t)_{t\geq 0}$ of the variable speed random walk. Conversely, suppose that $\varrho(\cdot)$ is such that \eqref{eq:measure-preserving} holds and that $\nu$ is given by \eqref{eq:disintegration}. Then we have for any $t \geq 0$ by Lemma \ref{lemma:martingale-property}, 
    \begin{align*}
        \int_{\mathcal{S}'(\Z^d)}\mathbb{E}_\phi\left[e^{i\langle u, \phi_t \rangle}\right]\nu(\dd\phi)
        =&
        \int_{\mathcal{S}'(\Z^d)} \varrho(\dd\xi) \int_{\mathcal{S}'(\Z^d)}\mu_\xi(\dd\phi)
        \exp\left(i \langle P_t u, \phi \rangle - \int_0^t\norm{P_s u}^2_2\dd s\right)
        \\
        =&
        \int_{\mathcal{S}'(\Z^d)}\exp\left(i \langle P_t u, \xi\rangle -\int_0^\infty \norm{P_s u}^2_2 \dd s \right)\varrho(\dd\xi)
        \\
        \overset{\eqref{eq:measure-preserving}}{=}&
        \int_{\mathcal{S}'(\Z^d)}\exp\left(i \langle u, \xi \rangle - \int_0^\infty \norm{P_s u}^2_2 ds\right)\varrho(\dd\xi)
        \\
        \overset{\text{def}}{=}& 
        C_\nu(u). 
    \end{align*}
    where we used the definition of $\mathcal{N}(\xi,G)$ for the second equality and the invariance of $\varrho$ with respect to the semigroup $(P_t)_{t\geq0}$ for the third equality. 
\end{proof}

With similar ideas, we can now show that every element of $\mathscr{G}(\gamma^\omega)$ is actually a reversible measure for \eqref{sde} and vice versa. The implication that every reversible measure is necessarily an element of $\mathscr{G}(\gamma^\omega)$ is not contained in \cite{holley_generalized_1978} and to our knowledge novel in this context. 

\begin{proof}[Proof of Proposition~\ref{lemma:gibbs-reversible}]
    First, let $\nu \in \mathscr{G}(\gamma^\omega)$. Calculating both sides separately gives 
    \begin{align*}
        &\int_{\mathcal{S}'(\Z^d)} \mathbb{E}_\phi\left[\exp\left(
        i \left(\langle u, \phi_t\rangle + \langle v, \phi_s\rangle\right)
        \right)\right]\nu(d\phi)
        \\\
        =
        &\int_{\mathcal{S}'(\Z^d)} 
        \exp\Big(i \langle P_t u + P_s v, \xi\rangle - \int_0^\infty \norm{P_r u}^2_2 \dd r \\  
        & \quad \quad \quad \quad \quad \enspace - \int_0^\infty \norm{P_r v}^2_2 \dd r - 2 \int_0^\infty \langle P_{t-s+r}u, P_{r} v\rangle  \dd r\Big)\varrho(\dd\xi)
    \end{align*}
    and 
    \begin{align*}
        &\int_{\mathcal{S}'(\Z^d)}  \mathbb{E}_\phi\left[\exp\left(
            i \left(\langle u, \phi_s\rangle + \langle v, \phi_t\rangle\right)
        \right)\right]\nu(\dd\phi)
        \\\
        =
        &\int_{\mathcal{S}'(\Z^d)} \exp\Big(i \langle P_s u + P_t v, \xi\rangle - \int_0^\infty \norm{P_r u}^2_2 \dd r 
        \\
        & \quad \quad \quad \quad \quad \enspace - \int_0^\infty \norm{P_r v}^2_2 \dd r - 2\int_0^\infty \langle P_r u, P_{t-s+r} v\rangle  \dd r\Big)\varrho(\dd\xi). 
    \end{align*}
    Now the symmetry of the semigroup $(P_t)_{t \geq 0}$ with respect to the dual pairing implies 
    \begin{align*}
        \int_0^\infty \langle P_{t-s+r}u, P_r v\rangle \dd r 
        =
        \int_0^\infty \langle P_r u, P_{t-s+r}v\rangle \dd r
    \end{align*}
    and under assumption \eqref{eq:strict-invariance} we have 
    \begin{align*}
        \langle P_t u, \xi \rangle = \langle P_s u, \xi\rangle 
        \quad 
        \text{and}
        \quad 
        \langle P_t v, \xi \rangle = \langle P_s v, \xi \rangle \quad 
        \varrho\text{-almost surely}. 
    \end{align*}
    Thus, every $\nu \in \mathscr{G}(\gamma^\omega)$ is a reversible measure for \eqref{sde}. 

    To see the converse, first note that by using the same steps as in the proof of Proposition~\ref{proposition:gaussian-disintegration}, every reversible \(\nu\) can be expressed as 
\begin{align}
    \nu = \int_{\mathcal{S}'} \, \mathcal{N}(\xi, G^\omega) \, \varrho(\mathrm{d} \xi),
\end{align} where \(\varrho\) is reversible with respect to \((P_t)_{t \geq 0}\), i.e.\, it satisfies
\begin{align}
    \varrho[F(P_t \psi) \, G(\psi)]
    = \varrho[F(\psi) \, G(P_t \psi)]
\end{align} 
for all non-negative measurable \(F, G\) and \(t \geq 0\).
Now let us show that this implies \(\varrho(\{\psi \,\colon\, P_t \psi = \psi \,\forall t \geq 0\}) = 1\), which then implies that \(\nu\) is Gibbs by the characterisation in Proposition \ref{proposition:gaussian-gibbs-measures}.
We have
\begin{align}
    \varrho[F(P_t \psi) \, G(\psi)]
    &= \varrho[F(P_{t/2} P_{t/2} \psi) \, G(\psi)] 
    = \varrho[F(P_{t/2} \psi) \, G(P_{t/2} \psi)] \\
    &= \varrho[(FG)(P_{t/2} \psi) \cdot 1]
    = \varrho[(FG)(\psi) \cdot P_{t/2} 1] \nonumber\\
    &= \varrho[F(\psi) \, G(\psi)] \nonumber
\end{align} 
for all measurable \(F, G \geq 0\). It follows that \(P_t \psi = \psi\) for \(\varrho\)-a.e.\ \(\psi\). The continuity of \(t \mapsto P_t \psi\), see Lemma~\ref{lemma:mapping-properties-tempered-distributions} and Remark~\ref{remark:C0_semigroup_on_H_minus_N_and_analyticity}, lets us upgrade this to \(\varrho(\{\psi \,\colon\, P_t \psi = \psi \,\forall t \geq 0\}) = 1\).
\end{proof}

Before we proceed to studying the semigroup $(P_t)_{t\geq 0}$ a bit more in detail, let us show that if Assumption \ref{assumption:green_function_is_tempered} does not hold, then there can be no stationary measures for the SDE \eqref{sde} that are supported on $\mathcal{S}'(\Z^d)$. 

\begin{prop}\label{prop:non-existence}
    There are \((T_t)_{t \geq 0}\)-invariant probability measures on \(\mathcal{S}'\) if and only if \(x \mapsto G(x, x)\) is an \(\mathcal{S}'\)-function.
\end{prop}
Theorem \ref{theorem:nonexistence-stationary} is a simple consequence of Proposition \ref{prop:non-existence}. 

\begin{proof}[Proof of Proposition~\ref{prop:non-existence}]
    First suppose that \(G(x, x) = \infty\) for some (equivalently, all) \(x \in \Cluster\), and suppose \(\mu\) is a \((T_t)_{t \geq 0}\)-invariant probability measure on \(\mathcal{S}'\).
    Then, the same calculation as in the proof of Proposition~\ref{proposition:gaussian-disintegration} shows
    \begin{align*}
        \mu[\e^{i r \psi(x)}]
        = \mu[T_t\e^{i r \cdot(x)}(\psi)]
        \xrightarrow[t \to \infty]{} 0
    \end{align*} for \(r \in \mathbb{R}\setminus\{0\}\). Hence, such a \(\mu\) can indeed not exist.

    Now suppose that \(G(x, x) < \infty\) for all (equivalently, some) \(x \in \Cluster(\omega)\), but that \(x \mapsto G(x, x)\) is not an \(\mathcal{S}'\)-function.
    Then, the same calculation as in the proof of Proposition~\ref{proposition:gaussian-disintegration} now just considering all probability measures simply on functions on \(\Cluster(\omega)\) as opposed to \(\mathcal{S}'\)-functions, shows that a \((T_t)_{t \geq 0}\)-invariant measure has to be a mixture of Gaussian processes \(X^{\psi} \sim \mathcal{N}(\psi, G)\) with covariance function \(G\) on \(\Cluster(\omega)\). Fix some deterministic function \(\psi \colon \Cluster \to \mathbb{R}\).
    Then we have
    \begin{align}\label{equation:dichotomy_H_N_deterministic_plus_Gaussian}
        \mathbb{P}(X^\psi \in H_{-N}) 
        = \mathbb{P}(\psi + X_0 \in H_{-N})
        = \begin{cases}
            1, &(x \mapsto \sqrt{G}(x, x)) \in H_{-N} \text{ and } \psi \in H_{-N}, \\
            0, &\text{otherwise},
        \end{cases}
    \end{align} so that, since G has the property of not being a function in \(\mathcal{S}' = \bigcup_{N \geq 0} H_{-N}\), we also have
    \begin{align*}
        \mathbb{P}(X^\psi \in \mathcal{S}') = 0.
    \end{align*} It follows that \(\mu\) cannot be a probability measure supported on \(\mathcal{S}'\).
    Let us finally show that \eqref{equation:dichotomy_H_N_deterministic_plus_Gaussian} indeed holds. That \(\mathbb{P}(\psi + X_0 \in H_{-N}) = 1\) when \((x \mapsto \sqrt{G}(x, x)) \in H_{-N}\) and \(\psi \in H_{-N}\) is clear, since then even 
    \begin{align*}
        \mathbb{E}[\Vert X_0 \Vert_{H_{-N}}^2]
        = \sum_{x} \jap{x}^{-2N} \mathbb{E}[\vert X_0(x) \vert^2]
        =  \sum_{x} \jap{x}^{-2N} G(x, x)
        = \Vert \sqrt{G} \Vert_{H_{-N}}^2
        < \infty.
    \end{align*}
    The other half follows from the fact that, for any exhaustion with finite subsets \(\Lambda_n \Subset \Cluster\), we have by the multivariate Gaussian distribution of \((X_0(x))_{x \in \Lambda_n}\) that
    \begin{align*}
        \mathbb{E}\Big[\exp\Big(-\sum_{x \in \Lambda_n} \jap{x}^{-2N} \vert \psi(x) + X_0(x) \vert^2\Big) \Big]
        \leq \big(1 + 2 \sum_{x \in \Lambda_n} \jap{x}^{-2N} G(x, x) \big)^{-1/2}
    \end{align*} so that, if \(\Vert \sqrt{G} \Vert_{H_{-N}} = \infty\),
    \begin{align*}
        \mathbb{E}[\exp(-\Vert X^\psi \Vert_{H_{-N}}^2)]
        \leq \liminf_{n \to \infty} \mathbb{E}\Big[\exp\Big(-\sum_{x \in \Lambda_n} \jap{x}^{-2N} \vert \psi(x) + X_0(x) \vert^2\Big) \Big]
        = 0.
    \end{align*}
\end{proof}

\section{Stationary measures for the heat semigroup}\label{sec:random-walk-semigroup}
In this section, we want to discuss the possibility of having \((P_t)_{t \geq 0}\)-stationary measures \(\varrho\) on \(\mathcal{S}'\) with the property that
\begin{align*}
    \varrho(\{\xi \in \mathcal{S}' \,\colon\,  \Delta \xi = 0 \}) 
    < 1.
\end{align*}
The main goals of this section are to prove Theorem \ref{theorem:stationary-iff-reversible}, Theorem \ref{theorem:non-reversible-periodic}, and Theorem \ref{theorem:spectral-characterisation}. 

We start by deriving some general results on the stationary measures for the random walk semigroup $(P_t)_{t\geq0}$ in Section \ref{sec:general-criteria} and then first prove Theorem \ref{theorem:stationary-iff-reversible} in the special case of uniformly bounded conductances in Section \ref{sec:bounded-weights} before proceeding to the case of unbounded conductances that grow subquadratic in Section \ref{sec:unbounded-weights}. The construction of a particular set of conductances for Theorem \ref{theorem:non-reversible-periodic} will be carried out in Section \ref{section:non-reversible-stationary-measures}.

\subsection{General criteria}\label{sec:general-criteria}
\subsubsection{Vanishing higher powers of the Laplacian}
\begin{lemma}[Support in the harmonic functions if the heat semigroup is nice]
    Suppose the heat semigroup \((P_t)_{t \geq 0}\) has the property that for every \(\psi \in \mathcal{S}'\) there is an \(m\) which only depends on \(N := \min\{M \geq 0 \colon \psi \in H_{-M} \}\) and such that
    \begin{align}
        (\Delta^m P_t \psi)(x)
        \xrightarrow[t \to \infty]{} 0, \quad \text{pointwise for all } x.
    \end{align} 
    Then,  every \((P_t)_{t \geq 0}\)-stationary probability measure \(\varrho\) on \(\mathcal{S}'\) satisfies
    \begin{align*}
        \varrho(\{\xi \in \mathcal{S}' \,\colon\,  \Delta \xi = 0 \}) 
        = \varrho(\{\xi \in \mathcal{S}' \,\colon\,  P_t \xi = \xi \,\forall t \geq 0 \}) 
        = 1.
    \end{align*}
\end{lemma}
\begin{proof}
    This is just a combination of Lemma~\ref{lemma:invariant_support_from_vanishing_higher_order_derivatives} and Lemma~\ref{lemma:vanishing_higher_order_derivatives_from_pointwise_convergence} below.
\end{proof}

\begin{lemma}[Invariant support from vanishing higher-order derivatives]\label{lemma:invariant_support_from_vanishing_higher_order_derivatives}
    Let \(\varrho\) be a \((P_t)_{t \geq 0}\)-stationary probability measure on \(\mathcal{S}'\).
    If it satisfies
    \begin{align}\label{equation:support_in_vanishing_under_higher_order_derivative}
        \varrho\big( \{\xi \in \mathcal{S}' \,\colon\, \exists m \in \mathbb{N} \text{ s.t. } \Delta^m \xi = 0 \} \big)
        = 1
    \end{align} then it automatically satisfies
    \begin{align*}
        \varrho(\{\xi \in \mathcal{S}' \,\colon\,  \Delta \xi = 0 \}) 
        = \varrho(\{\xi \in \mathcal{S}' \,\colon\,  P_t \xi = \xi \,\forall t \geq 0 \}) 
        = 1.
    \end{align*}
\end{lemma}
\begin{proof}
    By construction we have for all $\xi \in \mathcal{S}'(\Z^d)$ and $w\in \mathcal{S}$
    \begin{align}
        \langle P_t w, \xi \rangle - \langle w, \xi \rangle = - \int_0^t \langle \Delta P_s w, \xi \rangle \dd s, \quad t \geq 0. 
    \end{align}
    By induction one can show that 
    \begin{align}
        \langle P_tw,\xi\rangle-\langle w,\xi\rangle=\sum_{k=1}^{n-1}\frac{(-t)^k}{k!}\langle\Delta^kw,\xi\rangle+\frac{(-1)^n}{(n-1)!}\int_0^t(t-s)^{n-1}\langle\Delta^nP_sw,\xi\rangle\ \dd s.
    \end{align}
    If $\xi$ is such that $\Delta^n \xi \equiv 0$ for some $n \in \N$, then this implies that the left hand side, $\langle P_t w, \xi \rangle - \langle w, \xi \rangle$, is a polynomial in $t$ with vanishing constant term. 
    In particular, if there is an $s\geq 0$ such that $\langle P_s w, \xi \rangle \neq \langle w , \xi \rangle$, then we must necessarily have 
    \begin{align*}
        \abs{\langle P_t w, \xi \rangle - \langle w , \xi \rangle } \to \infty \quad \text{as } t \to \infty.  
    \end{align*}
    But since $\varrho(\cdot)$ is a time-stationary measure for $(P_t)_{t\geq 0}$, we actually have 
    $$
        \text{Law}\left(\langle P_t w, \xi \rangle\right) = \text{Law}\left(\langle w, \xi \rangle\right) 
    $$
    under $\varrho$. So necessarily $\langle P_t w, \xi \rangle = \langle w, \xi \rangle$ with $\varrho$-probability $1$ for all fixed $t \geq 0$. To move the quantifier over $t\geq 0$ inside, use that the map
    \begin{align*}
        [0,\infty) \times \mathcal{S} \ni (t, w) \mapsto P_t w \in \mathcal{S}
    \end{align*}
    is continuous and the fact that both $[0,\infty)$ and $\mathcal{S}$ are separable, to conclude that 
    \begin{align*}
        \varrho(\{\xi \in \mathcal{S}' \,\colon\,  \Delta \xi = 0 \}) 
        = \varrho(\{\xi \in \mathcal{S}' \,\colon\,  P_t \xi = \xi \,\forall t \geq 0 \}) 
        = 1,
    \end{align*}
    as claimed. 
\end{proof}

\begin{lemma}[Vanishing higher-order derivatives from pointwise convergence to zero]\label{lemma:vanishing_higher_order_derivatives_from_pointwise_convergence}
    Suppose the heat semigroup \((P_t)_{t \geq 0}\) has the property that for every \(\psi \in \mathcal{S}'\) there is an \(m\), which only depends on \(N := \min\{M \in \mathbb{N}_0 \colon \psi \in H_{-M} \}\), such that \eqref{equation:pointwise_converge_to_zero_higher_order_derivative} holds.
    Then the Assumption \eqref{equation:support_in_vanishing_under_higher_order_derivative} in Lemma~\ref{lemma:invariant_support_from_vanishing_higher_order_derivatives} is automatically satisfied.
\end{lemma}
\begin{proof}
    We have
    \begin{align}\label{equation:probability_of_non_vanishing_higher_order_derivative}
        \varrho\big( \{\psi \in \mathcal{S}' \,\colon\, \forall m \in \mathbb{N} \text{ it is } \Delta^m \psi \neq 0 \} \big)
        = \inf_{m \in \mathbb{N}} \sup_{R \geq 0} \sup_{\delta > 0} \varrho\big(\max_{x \in \Cluster \cap B_R} \vert \Delta^m \psi(x)\vert > \delta \big),
    \end{align} where here \(B_R := \{x \,\colon\, \jap{x} \leq R\}\).
    Define 
    \begin{align*}
        N(m) := \sup\{ N \geq 0 \,\colon\, \eqref{equation:pointwise_converge_to_zero_higher_order_derivative} \text{ holds for all } \psi\in H_{-N} \}
    \end{align*} and notice that \(N(m) \to \infty\) as \(m \to \infty\) by assumption. First suppose that $N(m)$ is finite for all $m$. 
    Now, for fixed \(m \geq 0\) it holds by stationarity, and \eqref{equation:pointwise_converge_to_zero_higher_order_derivative} together with bounded convergence, that
    \begin{align*}
        &\varrho\big(\max_{x \in \Cluster \cap B_R} \vert \Delta^m \psi(x)\vert > \delta \big)
        = \limsup_{t \to \infty} \varrho\big[\mathbf{1}_{(\delta, \infty)}(\max_{x \in \Cluster \cap B_R} \vert \Delta^m P_t \psi(x)\vert) \big] \\
        &\leq \limsup_{t \to \infty} \varrho\big[\mathbf{1}_{(\delta, \infty)}(\max_{x \in \Cluster \cap B_R} \vert \Delta^m P_t \psi(x)\vert)  \, \mathbf{1}_{H_{-N(m)}}(\psi) + \mathbf{1}_{(H_{-N(m)})^c}(\psi) \big] \\
        &= \varrho((H_{-N(m)})^c). 
    \end{align*} Hence, the r.h.s. of \eqref{equation:probability_of_non_vanishing_higher_order_derivative} vanishes, which in turn yields \eqref{equation:support_in_vanishing_under_higher_order_derivative}. If $N(m)$ may also already be equal to infinity for some finite $m$, but in that case  $H_{-\infty} = \emptyset$, so the same argument works. 
\end{proof}

\subsubsection{(Generalized) spectrum of the Laplacian}
Here we prove the spectral characterisation of the occurrence of non-fixing $(P_t)_{t\geq0}$-stationary measures given in Theorem \ref{theorem:spectral-characterisation}. 
\begin{proof}[Proof of Theorem~\ref{theorem:spectral-characterisation}]
    By \cite[Theorem]{Flytzanis1995} resp. \cite[Lemma 4.4]{GrivauxMartinez2023}, applied to the Hilbert space \(H_{-N}\) and \(P_T\) as a bounded linear map on it, and \(\varrho\) as \(P_T\)-invariant measure,  we have that
    \begin{align*}
        \mathrm{supp}(\varrho)
        \subseteq \overline{\mathrm{span}\{\psi \in H_{-N} \,\colon\, \text{there is a } \lambda \in \mathbb{C}, \vert \lambda\vert = 1 \text{ with } P_T \psi = \lambda \psi\}}.
    \end{align*}
    Hence, there exists \(\lambda = \e^{i \theta}\) and \(0 \neq \psi \in H_{-N}\) with \(P_T \psi = \lambda \psi\) and because \(\varrho\big(\{\xi \in \mathcal{S}' \,\colon\, P_T \xi = \xi \} \big) < 1\), we can choose it to be \(\lambda \neq 1\).

    That there is indeed a \(\theta \in \mathbb{R}\setminus\{0\}\) and \(\xi \in H_{-N}\) with \(P_t \xi = \e^{i\theta t} \xi\) for all \(t \geq 0\), or equivalently \(\Delta \xi = -\theta i  \xi\), is a consequence of spectral mapping with respect to the point spectrum of a \(C_0\) semigroup, cf. \cite[3.7 Spectral Mapping Theorem for Point and Residual Spectrum]{engel_one-parameter_2000}. That \((P_t)_{t \geq 0}\) is a \(C_0\)-semigroup on \(H_{-N}\) is proven in Remark~\ref{remark:C0_semigroup_on_H_minus_N_and_analyticity}.

    One can prove the converse using the same construction as in Lemma~\ref{lemma:purely_imaginary_eigenvectors_in_S_prime} below. 
\end{proof}

\begin{remark}[Support on \(H_{-N}\)]\label{remark:support_on_H_minus_N}
    Under only the assumption (1), one also gets the existence of some fixed \(N > 0\) such that 
    \begin{align*}
        \mu(A) = \frac{1}{\varrho(X_{-N})} \varrho(A \cap X_{-N})
    \end{align*} is well-defined for \(X_{-N} = \{\xi \in \mathcal{S}' \colon \mathrm{ord}(\xi) = N\}\), \(\mathrm{ord}(\xi) = \min\{M \geq 0 \,\colon\, \xi \in H_{-M}\}\), and such that \(\mu = \mu \circ P_T^{-1}\) and  \(\mu\big(\{\xi \in \mathcal{S}' \,\colon\, P_T \xi = \xi \} \big) < 1\). 
    This is essentially because \(P_T\) maps \(H_{-N}\) into \(H_{-N}\), which then forces \(\varrho(X_{-N} \cap P_T^{-1} X_{-N}) = \varrho(X_{-N})\) for all \(N \geq 0\).
    The same measure restriction then holds for \(H_{-N}\) in place of \(X_{-N} \subseteq H_{-N}\). 
    Hence, it is really just the moment \(\mu[\Vert \xi\Vert_{H_{-N}}^2] < \infty\), not the support assumption \(\mathrm{supp}(\mu) \subseteq H_{-N}\), that is missing to close the general case. Nevertheless, this is not only a technical problem; see \cite{Steenbeck2026} for why at least for general bounded linear operators on separable complex Hilbert spaces, non-trivial invariant probability measures do not have to imply existence of corresponding eigenvectors.
\end{remark}
\begin{remark}[Approximate point-spectrum of \(\Delta\) on \(H_{-N}\)]\label{remark:approximate_point_spectrum_of_Delta_on_H_minus_N}
    The existence of non-trivial \(P_T\)-invariant probability measures on \(H_{-N}\) implies the existence of purely-imaginary \emph{approximate} point spectrum of \(\Delta\) on \(H_{-N}\), i.e.\,\ \(\sigma_{\mathrm{ap}}(\Delta; H_{-N}) \cap i \mathbb{R}\setminus\{0\} \neq \emptyset\).
    Indeed, by Remark~\ref{remark:C0_semigroup_on_H_minus_N_and_analyticity}, \((P_t)_{t \geq 0}\) is even an analytic semigroup, so that the spectral mapping theorem \cite[Chapter IV, Corollary 3.12]{engel_one-parameter_2000} for approximate point-spectrum holds:
    \begin{align*}
        \sigma_{\mathrm{ap}}(P_t; H_{-N}) \setminus\{0\}
        = \e^{-t \,\sigma_{\mathrm{ap}}(\Delta; H_{-N})}, \quad t \geq 0.
    \end{align*}
    In particular, if \(\Delta\) on \(H_{-N}\) has no purely-imaginary spectrum \(\sigma_{\mathrm{ap}}(\Delta; H_{-N}) \cap i \mathbb{R}\setminus\{0\} = \emptyset\), it follows that \(\sigma_{\mathrm{ap}}(P_T; H_{-N}) \,\cap\, \mathbb{S}^1 \setminus\{1\} = \emptyset\).
    Now we can apply the general theorem \cite[Theorem 1.5]{Steenbeck2026} that a bounded linear map on a separable complex Banach space can only have non-fixing invariant probability measures if it has approximate point-spectrum on the unit circle \(\mathbb{S}^1\) minus \(1\).

    At the moment, we don't know if the approximate point-spectrum can be used to construct non-fixing stationary measures or if there automatically has to be purely-imaginary point spectrum for \(\Delta\).
\end{remark}

\subsection{Bounded weights}\label{sec:bounded-weights}

Here, we will give a somewhat probabilistic proof for the fact that there are no non-fixing stationary measures for \((P_t)_{t \geq 0}\) if the conductances are uniformly bounded. The statement itself also follows from the more general Theorem~\ref{theorem:stationary_measures_under_subquadratic_bounds}, where subquadratic conductances are allowed. Nevertheless, the corresponding proof draws on the same ideas of balancing decay of powers of the Laplacian versus growth of moments in time, despite being more analytic in nature, so that we may present them in a simpler setting first.
\begin{proof}
    We will see that under bounded conductances, the criterion Lemma~\ref{lemma:support_in_harmonic_functions_if_heat_semigroup_is_nice} is fulfilled, i.e.\, that for every \(\psi \in \mathcal{S}'\) there is an \(m\) which only depends on \(N := \min\{M \geq 0 \colon \psi \in H_{-M} \}\) and such that
    \begin{align}\label{equation:pointwise_converge_to_zero_higher_order_derivative_copy}
        (\Delta^m P_t \psi)(x)
        \xrightarrow[t \to \infty]{} 0, \quad \text{pointwise for all } x.
    \end{align} 
    
    Now, of course, one can always represent the variable-speed random walk \((X_t)_{t \geq 0}\) via the discrete-time random walk with transition probabilities proportional to the edge conductances and having exponential waiting time with parameter \(\sum_{y \sim x} \omega(x, y)\)  
    at the vertex \(x\) until the next jump, but it will be helpful for us to separate path and waiting times a little bit.
    For that, let \(\mathfrak{w} := \sup_x \sum_{y \sim x} \omega(x, y)\) and consider the lazy discrete-time random walk \((Y_n)_{n \in \mathbb{N}_0}\) which has one-step transition probabilities
    \begin{align*}
        P(Y_{n+1} = y \,\vert\, Y_n = x)
        = \begin{cases}
            \frac{\omega(x, y)}{\mathfrak{w}}, & y \sim x, \\
            1 - \frac{\sum_{z \sim x} \omega(x, z)}{\mathfrak{w}}, &y = x.
        \end{cases}
    \end{align*} Then, we can reconstruct (the distribution of) \((X_t)_{t \geq 0}\) by having \((Y_n)_{n \in \mathbb{N}_0}\) as embedded Markov chain and having a mean waiting time \(1 / \mathfrak{w}\) between jumps.
    That means we can write \(X_t = Y_{N_t}\) for a \(\mathrm{Poisson}(\mathfrak{w}t)\)-distributed random variable \(N_t\) independent of \(Y\), in particular it holds for \(\psi \in \mathcal{S}'\) that
    \begin{align*}
        (P_t \psi)(x)
        = E_x[\psi(Y_{N_t})]
        = \sum_{n = 0}^{\infty} \e^{- \mathfrak{w} t} \frac{(\mathfrak{w}t)^n}{n!} E_x[\psi(Y_n)],
    \end{align*} 
    and more generally
    \begin{align}\label{eqn:poissonisation}
        ((-\Delta)^m P_t \psi)(x)
        = \Big(\frac{\dd^m}{\dd t^m} P_t \psi \Big) (x)
        = \sum_{n = 0}^{\infty} \Big( \frac{\dd^m}{\dd t^m} \e^{- \mathfrak{w} t} \frac{(\mathfrak{w}t)^n}{n!}\Big) E_x[\psi(Y_n)]
    \end{align} for \(m \geq 0\).
    
    Let us now derive nicer expressions for the derivatives of the Poisson weights \(\frac{\dd^m}{\dd t^m} \e^{- t} \frac{t^n}{n!}\). These can be conveniently expressed in terms of the Charlier orthogonal polynomials, see e.g.\ \cite[Chapter 10]{peccati_wiener_2011} or \cite[Section 3]{last_poisson_2011}, given by the recurrence relation $c_0(n,t) = 1$ and 
    \begin{equation}\label{eqn:charlier-recurrence}
        c_{m+1}(n,t) = \frac{n}{t}c_m(n-1,t) - c_m(n,t). 
    \end{equation}
    They exhibit the following nice orthogonality relation, that if $M_t$ is a Poisson random variable with intensity $t > 0$, then for any $m,k\in \N_0$
    \begin{equation}
        \mathbb{E}[c_{m}(M_t,t)c_k(M_t,t)] = \delta_m(k) \frac{m!}{t^m}. 
    \end{equation}
    Introducing the notation $\pi_t(n) = e^{-t}\tfrac{t^n}{n!}$ for the Poissonian weights and using the simple observation that 
    \begin{equation}
        \partial_t \pi_t(n) = \pi_t(n-1) - \pi_t(n), 
    \end{equation}
    one can check via the recurrence relation \eqref{eqn:charlier-recurrence} that the Charlier polynomials can be expressed as
    \begin{equation}
        c_m(n,t) = \frac{\partial_t^m \pi_t(n)}{\pi_t(n)}. 
    \end{equation}
    This means we can rewrite \eqref{eqn:poissonisation} as 
    \begin{equation}\label{eqn:charlier-representation}
        ((-\Delta)^m P_t \psi)(x)
        = \mathfrak{w}^m E_x[c_m(N_t, \mathfrak{w}t) \psi(X_t)].
    \end{equation}
    Now we want to use Cauchy--Schwarz to get 
    \begin{equation}
        \abs{(\Delta^m P_t \psi)(x)} 
        \leq  \mathfrak{w}^m E_x[c_m(N_t, \mathfrak{w}t)^2]^{1/2} E_x[\psi(X_t)^2]^{1/2} 
        = \frac{\mathfrak{w}^{m/2} (m!)^{1/2}}{t^{m/2}} \mathbb{E}[\psi(X_t)^2]^{1/2}.  
    \end{equation}
    It remains to bound the random walk expectation, but this is easily bounded by $\sim (1+ \abs{x} + t)^N$ when \(\psi \in H_{-N}\). So we can choose $m > 2N$ to get the claimed decay \eqref{equation:pointwise_converge_to_zero_higher_order_derivative_copy}.
\end{proof}

\subsection{Subquadratically bounded weights}\label{sec:unbounded-weights}

Theorem~\ref{theorem:stationary_measures_under_subquadratic_bounds} follows from Lemma~\ref{lemma:support_in_harmonic_functions_if_heat_semigroup_is_nice} and the following
\begin{prop}
    Assume that the conductances $\omega$ satisfy Assumption~\ref{assumption:subquadratic-growth}.
    Then, the heat semigroup \((P_t)_{t \geq 0}\) has the property that \eqref{equation:pointwise_converge_to_zero_higher_order_derivative} holds for every \(m > \tfrac{2}{\epsilon} N\).
\end{prop}
\begin{proof}
    Consider \(\psi \in H_{-N}\) and \(m > \frac{2N}{\epsilon}\).
    Then, for every x,
    \begin{align*}
        (\Delta^m P_t \psi)(x)
        = \langle \Delta^m P_t \psi, \mathbf{1}_x \rangle_{\mathcal{S}', \mathcal{S}}
        = \langle \psi, P_t \Delta^m \mathbf{1}_x \rangle_{\mathcal{S}', \mathcal{S}}
        \leq \Vert\psi\Vert_{H_{-N}} \Vert P_t \Delta^m \mathbf{1}_x\Vert_{H_{N}}
    \end{align*} and the r.h.s. converges to zero as \(t \to \infty\) by Lemma~\ref{lemma:H_N_convergence_to_zero_of_higher_order_derivative}.
\end{proof}

For the last proof, we needed the following
\begin{lemma}\label{lemma:H_N_convergence_to_zero_of_higher_order_derivative}
    Assume that the conductances $\omega$ satisfy Assumption~\ref{assumption:subquadratic-growth}.
    Then, for every \(u \in C_c\), we have
    \begin{align}
        \Vert P_t \Delta^m u \Vert_{H_{N}}
        \xrightarrow[t \to \infty]{} 0
    \end{align} for all \(N \geq 0\) and \(m > \tfrac{2}{\epsilon} N\).
\end{lemma}
\begin{proof}
    Let us first see that the same decay holds in the weaker \(\ell^2\)-norm in place of the \(H_{N}\)-norm.
    Note that \(C_c\) is a subset of the \(\ell^2\)-domain of all powers of \(\Delta = \Delta^{(D)}\) and hence
    \begin{align}
        \Vert P_t \Delta^m u \Vert_{\ell^2}
        = \Vert \Delta^m \mathrm{e}^{-t \Delta} u \Vert_{\ell^2}
        \leq  \Vert \Delta^m \mathrm{e}^{-t \Delta} \Vert_{\ell^2 \to \ell^2} \, \Vert u \Vert_{\ell^2}.
    \end{align}
    The spectral theorem for self-adjoint operators implies that the operator norm on the r.h.s. is bounded by
    \begin{align}
         \Vert \Delta^m \mathrm{e}^{-t \Delta} \Vert_{\ell^2 \to \ell^2} 
         \,\leq\, \sup_{\lambda \geq 0} \lambda^m \mathrm{e}^{-t \lambda}
         = \Big(\frac{m}{\mathrm{e} t}\Big)^m.
    \end{align}
    We can transfer this decay to the \(H_{N}\)-norm via Cauchy--Schwarz
    \begin{align*}
        \Vert P_t \Delta^m u \Vert_{H_N}^2
        \leq \Vert P_t \Delta^m u \Vert_{\ell^2} \, \Vert P_t \Delta^m u \Vert_{H_{2N}},
    \end{align*} if we know that \(\Vert P_t \Delta^m u \Vert_{\ell^2}\) has a faster decay to zero than \(\Vert P_t \Delta^m u \Vert_{H_{2N}}\) can grow as \(t \to \infty\).
    This is ensured by Lemma~\ref{lemma:polynomial_in_time_operator_norm_bounds} if we choose \(m > \tfrac{2}{\epsilon} N\).
\end{proof}

\begin{lemma}\label{lemma:polynomial_in_time_operator_norm_bounds}
    Assume that the conductances $\omega$ satisfy Assumption~\ref{assumption:subquadratic-growth}.
    Then,
    \begin{align}\label{equation:polynomial_in_time_H_N_operator_norm_bound}
        \Vert P_t \Vert_{H_{N} \to H_{N}}
        \lesssim_N (1 + t)^{N / \epsilon},
    \end{align} for all \(N \geq 0\).
\end{lemma}
\begin{proof}
    Recall that Lemma~\ref{lemma:semigroup_norm_bound_schwartz} gave us the \emph{exponential} in \(t\) bound
    \begin{align*}
        \Vert P_t \Vert_{H_{N} \to H_{N}}
        \leq \mathrm{e}^{K_N t/2}
    \end{align*} for quadratically bounded weights \(\omega\) and with constant
    \begin{align*}
        K_N
        := \sup_x \,\, \jap{x}^{-2N} \sum_{y \sim x} \omega(x, y) \vert \jap{x}^N - \jap{y}^N \vert^2.
    \end{align*}
    As we assume that our weights are only allowed to grow like \(\sim \jap{x}^{2-\epsilon}\), there is some room to optimize. 
    Indeed, let us define a new norm by
    \begin{align*}
        \Vert u\Vert_{H_{N, \alpha}}^2
        := \sum_{x} \jap{x}_\alpha^{2N} \vert u(x)\vert^2
    \end{align*} with  \(\jap{x}_\alpha := (\vert x\vert^2 + \alpha^2)^{1/2}\) for \(\alpha \geq 1\).
    Then the same proof shows that
    \begin{align*}
        \Vert P_t \Vert_{H_{N, \alpha} \to H_{N, \alpha}}
        \leq \mathrm{e}^{K_{N, \alpha} t/2}
    \end{align*} with
    \begin{align*}
        K_{N, \alpha}
        := \sup_x \,\, \jap{x}_\alpha^{-2N} \sum_{y \sim x} \omega(x, y) \vert \jap{x}_\alpha^N - \jap{y}_\alpha^N \vert^2.
    \end{align*} and hence
    \begin{align}\label{equation:operator_norm_bound_with_alpha_to_be_optimized}
        \Vert P_t \Vert_{H_{N} \to H_{N}}
        \leq \alpha^N \Vert P_t \Vert_{H_{N, \alpha} \to H_{N, \alpha}}
        \leq \alpha^N \mathrm{e}^{K_{N, \alpha} t/2}.
    \end{align}
    Let us now first use that the conductances grow subquadratically by Assumption~\ref{assumption:subquadratic-growth}  to control \(K_{N, \alpha}\) and then optimize in \(\alpha\).
    Observe that for nearest-neighbours \(x \sim y\), 
    \begin{align*}
        \vert \jap{x}_\alpha^N - \jap{y}_\alpha^N \vert^2
        \leq C_N' \jap{x}_\alpha^{2(N - 1)}
    \end{align*} by the mean-value theorem.
    Hence,
    \begin{align*}
        K_{N, \alpha}
        \leq C_N' \sup_x \Big(\sum_{y \sim x} \omega(x, y) \Big) \jap{x}_\alpha^{-2}
        \leq C_N'' \sup_{x} \jap{x}_\alpha^{-\epsilon}
        = C_N''  \alpha^{-\epsilon}
    \end{align*} by Assumption~\ref{assumption:subquadratic-growth}.
    By \eqref{equation:operator_norm_bound_with_alpha_to_be_optimized} we further get
    \begin{align*}
        \Vert P_t \Vert_{H_{N} \to H_{N}}
        \leq \alpha^N \exp( C_N''  \alpha^{-\epsilon} t / 2).
    \end{align*}
    Setting \(\alpha = (1+t)^{1/\epsilon}\) finally yields \eqref{equation:polynomial_in_time_H_N_operator_norm_bound}.
\end{proof}

\subsection{Quadratically growing weights: non-reversibility and periodicity}\label{section:non-reversible-stationary-measures}
\begin{proof}[Proof of Theorem~\ref{theorem:non-reversible-periodic}]
Consider the one-dimensional line \(\mathbb{Z}\) equipped with conductances \(\omega(n, n+1) = n^2 + 1\) for \(n \in \mathbb{Z}\). One could think that the variable speed random walk on \(\Z\) with symmetric these conductances is explosive, but it turns out to be barely non-explosive. On \(\Z\) with nearest neighbour conductances, \cite[Theorem 9.25]{keller2021graphs} shows that the variable speed random walk with conductances \(c\) is non-explosive if and only if
\begin{align*}
    \sum_{n \in \Z} \frac{\vert n\vert}{c(n, n+1)}
    = \infty,
\end{align*} 
which is satisfied for our choice. 
    
Moreover, the total conductances at a single site are obviously not growing faster than quadratically, so Assumption \ref{assumption:at_most_quadratically_growing}  is satisfied. Additionally, we also have that \(x \mapsto G(x, x)\) is an \(\mathcal{S}'\)-function. In fact, it is even uniformly bounded by the constant \(\frac{1}{4} \sum_{n \in \mathbb{Z}} \frac{1}{n^2 + 1}\). In particular, the zero boundary condition \(\varphi\)-Gibbs measure has its support in \(\mathcal{S}'\).
    
To show that for this choice of conductances not every time-stationary probability measure is reversible, i.e.\, a Gibbs measure, according to Proposition~\ref{proposition:gaussian-disintegration} and the characterisation of Gibbs measures in Proposition \ref{proposition:gaussian-gibbs-measures}, it suffices to construct some probability measure \(\varrho\) on \(\mathcal{S}'\) such that \(\varrho \circ P_t^{-1} = \varrho\) for all \(t \geq 0\), but
\begin{align*}
        \varrho\big(\big\{\xi \,\colon\, P_t \xi = \xi \,\,\forall t \geq 0 \big\} \big)
        = 0.
\end{align*}
For that, we first construct $\mathcal{S}'(\Z^d)$-valued time-periodic solutions to the discrete heat equation $\partial_t u = -\Delta_\omega u$. 
\begin{lemma}\label{lemma:purely_imaginary_eigenvectors_in_S_prime}
    Consider the conductances $\omega(n,n+1)=\omega(n+1, n) =n^2 + 1$ for $n\in \Z$. There exist linearly independent \(\xi, \psi \in \mathcal{S}'\) such that
    \begin{align*}
        &P_t \xi
        = \cos(t) \xi + \sin(t) \psi, \\
        &P_t \psi
        = -\sin(t) \xi + \cos(t) \psi
    \end{align*} for all \(t \geq 0\).
    In particular, for
    \begin{align*}
        \psi_\theta := \cos(\theta) \xi + \sin(\theta) \psi
    \end{align*} it holds that
    \begin{align*}
        P_t \psi_\theta
        = \cos(\theta + t) \xi + \sin(\theta + t) \psi
        = \psi_{(\theta + t) \,\mathrm{mod}\, 2 \pi}
    \end{align*} and \(\psi_\theta \neq \psi_{\widetilde{\theta}}\) for distinct \(\theta, \widetilde{\theta} \in [0, 2\pi)\) by the linear independence of \(\xi, \psi\).
\end{lemma}

\begin{proof}
    We will show that there is a complex-valued \(0 \neq h \in \mathcal{S}'_{\mathbb{C}}\), where everything works just analogously to the real-valued case, such that \(P_t h = \mathrm{e}^{-it} h\) because of \(\Delta h = i h\). Then we can just take real and imaginary parts, \(\xi = \mathrm{Re}\, h\) and \(\psi = \mathrm{Im}\, h\). 

    Now, set \(h(0) := 1\), \(h(1) := 1\), and recursively
    \begin{align*}
        h(n+1) 
        := h(n) - \frac{i}{n^2 + 1} \sum_{k = 1}^{n} h(k), \quad n \geq 1.
    \end{align*} On the other half of \(\mathbb{Z}\), we set
    \begin{align*}
        h(n) := h(1-n), \quad n \leq 0,
    \end{align*} so that \(h\) is symmetric about \(1/2\).
    Let us explicitly check that \(\Delta h = i h\) pointwise. Indeed, for \(n \geq 1\), we have
\begin{align*}
    (\Delta h)(n)
    &= c(n, n+1)[h(n) - h(n+1) ] + c(n-1, n) [h(n) - h(n-1)] \\
    &= (n^2 + 1) \frac{i}{n^2 + 1} \sum_{k = 1}^{n} h(k) + ((n-1)^2 + 1) \frac{-i}{(n-1)^2 + 1} \sum_{k = 1}^{n-1} h(k) \\
    &= i  h(n).
\end{align*}
The case \(n \leq 0\) follows by symmetry from the case \(n \geq 1\). Clearly, \(h \in \mathcal{S}'\), as it grows at most linearly.
\end{proof} 

Once we have the $\mathcal{S}'(\Z^d)$-valued time-periodic orbits constructed in Lemma \ref{lemma:purely_imaginary_eigenvectors_in_S_prime} at our disposal, it suffices to note that we can put
\[\varrho = (2\pi)^{-1}\int_{0}^{2 \pi} \delta_{\psi_\theta} \dd\theta\]
with \(\psi_\theta\) from Lemma~\ref{lemma:purely_imaginary_eigenvectors_in_S_prime}. Then the measure $\nu$ defined via $\varrho$ in terms of \eqref{eq:disintegration} is a non-reversible stationary measure for \eqref{sde} by combining Proposition~\ref{proposition:gaussian-disintegration} and Proposition~\ref{lemma:gibbs-reversible}.

To get a non-trivial time-periodic orbit $(\nu_t)_{0\leq t\leq \tau}$ of measures on $\mathcal{S}'(\Z^d)$ it suffices to note that for any Gaussian measure $\mathcal{N}(\xi, G) = \xi + \mathcal{N}(0,G)$, $\xi \in \mathcal{S}'(\Z^d)$ the dynamics of the SDE acting on $\mathcal{N}(\xi, G)$ is particularly simple. Indeed, if $\phi_0 \sim \mathcal{N}(\xi, G)$, then by linearity for any $t\geq 0$ we have $\phi_t \sim \mathcal{N}(P_t \xi, G)$. So choosing $\nu_t := \mathcal{N}(P_t \psi, G)$ for $t\geq 0$ with $\psi$ from Lemma \ref{lemma:purely_imaginary_eigenvectors_in_S_prime} yields the desired non-trivial $2\pi$-periodic orbit. 
\end{proof}

\subsection{The intrinsic metric and a Shnol-type theorem}

\begin{proof}[Proof of Proposition~\ref{proposition:shnol_type_theorem}]
    According to the version of Shnol's theorem which is presented in \cite[Theorem 12.25]{keller2021graphs}, we have the following (although it is stated there for \(\lambda \in \mathbb{R}\) only, mutatis mutandis the same proof works).
    For every \(\lambda \in \mathbb{C}\) such that a function \(\psi \neq 0\) exists with
    \begin{enumerate}
        \item \(\Delta \psi = \lambda \psi\),
        \item \(\e^{-\alpha \rho(0, \cdot)} \psi \in \ell^2_\mathbb{C}\) for all \(\alpha > 0\),
    \end{enumerate} it is \(\lambda \in \sigma(\Delta^{(D)}; \ell^2)\), where \(\Delta^{(D)}\) is the (self-adjoint) Dirichlet Laplacian and \(\sigma(\Delta^{(D)}; \ell^2) \subseteq [0, \infty)\) its spectrum as unbounded operator on \(\ell^2_\mathbb{C}\). Condition (2) holds for every \(\psi \in \mathcal{S}'\)  under Assumption \ref{assumption:intrinsic_metric} on the intrinsic metric \(\rho\). Note that the intrinsic metric \(\rho\) we use has finite jump size \(\sup\{\rho(x,y) \,\colon\, x, y \text{ with } x \sim y \} < \infty\) as \(\rho(x, y) \leq 1\) for \(x \sim y\).
\end{proof}

\subsection{Random conductances} 
As a first step towards proving Theorem~\ref{theorem:random-conductances} and Theorem~\ref{theorem:random-conductances-2}, we need to make sure that Assumption~\eqref{assumption:at_most_quadratically_growing} is satisfied \(\mathbb{P}\)-almost surely. Therefore, we start with the following elementary observation.
\begin{lemma}[(Sub-)quadratic growth from moment bounds]\label{lemma:random_conductance_growth_from_moments}
    Assume that \(\mathbb{P}\) is translation-invariant and the moment assumption
    \begin{align}
        \mathbb{E}[\omega^{(d/2) + \delta}(x, y)] < \infty, \quad x,y \in \Z^d,
    \end{align} holds.
    Then, \(\mathbb{P}\)-a.s.\ it holds for the conductances \(\omega\) the at most quadratic growth \eqref{eqn:assumption-quadratic-growth} if \(\delta = 0\), and if \(\delta > 0\) it \(\mathbb{P}\)-a.s.\ holds \eqref{assumption-eqn:at_most_subquadratically_growing}, i.e.\, growth of order at most \(\jap{x}^{2-\epsilon}\), with \(\epsilon = \frac{2}{1 + (d / \delta)}\). 
\end{lemma}
\begin{proof}
    Let \(e_1, \dots, e_d\) be the canonical basis of \(\R^d\) and \(e_j = -e_{j-d}\) for \(d+1 \leq j \leq 2d\). Then, for all \(1 \leq j \leq 2d\) and \(\delta \geq 0\),
    \begin{align*}
        &\sum_{x \in \Z^d} \mathbb{P}(\omega(x, x + e_j) \geq \jap{x}^{2-\epsilon})
        = \sum_{x \in \Z^d} \mathbb{P}(\omega(0, e_j) \geq \jap{x}^{2-\epsilon}) \\
        &\lesssim \int_{0}^{\infty} \dd r \, r^{d-1} \, \mathbb{P}(\omega(0, e_j)^{1/(2-\epsilon)} \geq r) 
        \sim \mathbb{E}[\omega(0, e_j)^{d/(2-\epsilon)}].
    \end{align*} Hence, it follows by the Borel--Cantelli lemma that \(\mathbb{P}\)-a.s.\ 
    \begin{align*}
        \sup_x \sum_{y \sim x} \omega(x, y) \jap{x}^{-(2-\epsilon)} 
        \leq \sum_{j = 1}^{2d} \sup_x \omega(x, x + e_j) \jap{x}^{-(2-\epsilon)} 
        < \infty.
    \end{align*}
\end{proof}

To prove the statement about \(\mathbb{P}\)-a.s.\ non-existence of non-fixing \((P_t)_{t\geq 0}\)-invariant probability measures on \(\mathcal{S}'\) with certain moments in Theorem~\ref{theorem:random-conductances-2}, we want to apply Theorem~\ref{theorem:spectral-characterisation}. Thereby it is enough to see that there are \(\mathbb{P}\)-a.s.\ no purely imaginary eigenvalues of \(\Delta = \Delta_\omega\) with eigenvectors in \(\mathcal{S}'\). This is indeed \(\mathbb{P}\)-a.s.\ the case by the Shnol-type theorem Proposition~\ref{proposition:shnol_type_theorem} as soon as the Assumption~\ref{assumption:intrinsic_metric} on the intrinsic metric on \(\Cluster\) is fulfilled.

\begin{lemma}\label{lemma:assumption_intrinsic_metric_is_as_fulfilled_under_moments}
    Suppose \(\mathbb{P}\) is ergodic, translation-invariant and \(\mathbb{E}[\omega(x, y)^p] < \infty\), \(x, y \in \Z^d\), for \(p > \frac{d-1}{2}\). Then, \(\mathbb{P}\)-a.s.\
    the Assumption ~\ref{assumption:intrinsic_metric} is satisfied.
\end{lemma}
\begin{proof}
    The function \(\e^{-\alpha \rho^{\omega}(o, \cdot)}\) has almost surely stretched exponential decay for all \(\alpha > 0\), as
    \begin{align}\label{equation:lower_bound_intrinsic_metric}
        \rho(o,x)
        \gtrsim \langle x \rangle^{1-\frac{d-1}{2p}}
    \end{align} for \(x\) with \(\langle x \rangle \geq M_\omega\) with random constant \(M^\omega > 0\).
    This has essentially the same proof as \cite[Theorem 2.4]{AndresDeuschelSlowik2019}, which then gives
    \begin{align}
        \rho(o,x)
        \gtrsim d(o, x)^{1-\frac{d-1}{2p}}
    \end{align} for the graph distance \(d\) and all \(x\) with \(d(o, x) \geq N^\omega\), with some random constant \(N^\omega > 0\).
    Note that this readily implies \eqref{equation:lower_bound_intrinsic_metric} because the graph distance on the cluster \(\Cluster(\omega)\) is at least as large as the graph distance on \(\Z^d\).
\end{proof}

Putting our observations together, we get a proof of our two results for random conductances. 
\begin{proof}[Proof of Theorem~\ref{theorem:random-conductances} and Theorem~\ref{theorem:random-conductances-2}]
    That Assumption~\ref{assumption:at_most_quadratically_growing} is met \(\mathbb{P}\)-almost surely is proven in Lemma~\ref{lemma:random_conductance_growth_from_moments}.
    It also follows from Lemma~\ref{lemma:random_conductance_growth_from_moments} that the assumptions of Theorem~\ref{theorem:stationary_measures_under_subquadratic_bounds} are satisfied if \(\delta > 0\), so that in this case there will be no non-fixing \((P_t)_{t \geq 0}\)-invariant probability measures on \(\mathcal{S}'\). For the case of \(\delta = 0\), we already gave right before Lemma~\ref{lemma:assumption_intrinsic_metric_is_as_fulfilled_under_moments} the full argument for why there are \(\mathbb{P}\)-a.s.\ no non-fixing \((P_t)_{t\geq 0}\)-invariant probability measures on \(\mathcal{S}'\) with \(\varrho[\Vert \psi\Vert_{H_{-N}}^2] < \infty\) for some \(N \geq 0\).

    The $\mathbb{P}$-almost sure well-posedness of the SDE \eqref{sde} follows from Lemma~\ref{lemma:random_conductance_growth_from_moments} and Theorem~\ref{theorem:well-posed}. 
\end{proof}

\subsection*{Acknowledgements}
JDD and JK thank Marek Biskup, Martin Slowik, and Simone Warzel for helpful discussions on this and related problems. YS thanks Sebastian Andres and Leonid Kolesnikov for conversations about random walks and Gaussian free fields.
JK gratefully acknowledges the financial support of the Leibniz
Association within Benedikt Jahnel's Leibniz Junior Research Group on \textit{Probabilistic Methods for Dynamic Communication Networks} as part of the Leibniz Competition, hosted at the Weierstrass Institute Berlin.

\renewcommand*{\bibfont}{\footnotesize}
\printbibliography 

\end{document}